\documentclass[12pt]{amsart}
\usepackage{amsmath}
\usepackage{amscd}
\usepackage{pb-diagram}
\usepackage{comment}
\usepackage{amssymb}
\usepackage{graphicx}
\usepackage{subcaption}
\usepackage{color}
\usepackage{hyperref}

\usepackage{amssymb}

\def\refeq#1{\if\workingver y(\ref{#1})-[[#1]]\else(\ref{#1})\fi}
\def\refth#1{\if\workingver y\ref{#1}-[[#1]]\else\ref{#1}\fi}
\def\mylabel#1{\if\workingver y\label{#1}{\bf\ \ [[#1]]\ \ }\else\label{#1}\fi}
\def\mybibitem#1{\if\workingver y\bibitem{#1}{\bf\ \ [[#1]]\ \
}\else\bibitem{#1}\fi}

\def\articletheorems{
\newtheorem{thm}{Theorem}[section]
\newtheorem{lem}[thm]{Lemma}

\newtheorem{defn}[thm]{Definition}

\newtheorem{prop}[thm]{Proposition}

\newtheorem{ex}[thm]{Example}
\newtheorem{algo}{Algorithm}[section] 
\newtheorem{alg}[thm]{Algorithm}  
\newtheorem{rem}[thm]{Remark}
}

\def\map{\rightarrow}

\newcommand{\mto}{\multimap}
\newcommand{\pto}{\nrightarrow}

\renewcommand{\emptyset}{\varnothing}
\renewcommand{\rho}{\varrho}
\renewcommand{\phi}{\varphi}
\renewcommand{\epsilon}{\varepsilon}

\def\cC{\text{$\mathcal C$}}

\def\cE{\text{$\mathcal E$}}

\newcommand{\cl}{\operatorname{cl}}
\newcommand{\opn}{\operatorname{opn}}

\newcommand{\inte}{\operatorname{int}}

\newcommand{\dom}{\operatorname{dom}}

\newcommand{\im}{\operatorname{im}}

\renewcommand{\emptyset}{\varnothing}

\newcommand{\Inv}{\operatorname{Inv}}

\newcommand{\Endo}{\operatorname{Endo}}
\newcommand{\Auto}{\operatorname{Auto}}
\newcommand{\Mod}{\operatorname{Mod}}

\def\begeq#1{\begin{equation}\mylabel{#1}}
\def\endeq{\end{equation}}

\def\mathobj#1{\mbox{$#1$}}

\def\NN{\mathobj{\mathbb{N}}}

\def\QQ{\mathobj{\mathbb{Q}}}
\def\RR{\mathobj{\mathbb{R}}}

\def\ZZ{\mathobj{\mathbb{Z}}}

\def\rep#1{\mbox{$\langle#1\rangle$}}

\def\setof#1{\mbox{$\{\,#1\,\}$}}

\def\0#1{\hbox{\kern25pt}$ #1 $\\}
\def\1#1{\hbox{\kern40pt}$ #1 $\\}
\def\2#1{\hbox{\kern55pt}$ #1 $\\}
\def\3#1{\hbox{\kern70pt}$ #1 $\\}

\newcounter{li}

\def\begalg#1{\begin{algo}\mylabel{#1}\normalshape:\small\baselineskip 10pt\\}
\def\endalg{\end{algo}}

\def\Figures(include=#1,cat=#2){
  \renewcommand{\textfraction}{.20}
  \renewcommand{\topfraction}{.80}
  \renewcommand{\bottomfraction}{.80}
  \renewcommand{\floatpagefraction}{.80}
  \newcount\figcount
  \figcount=0
  \let\includefigures=#1
  \def\figcat{#2}
}

\def\FigureFromFile[#1][#2](#3)#4
{
  \begin{figure}[htbp]
     \global\advance\figcount by 1
     \if\includefigures y\special{anisoscale #1.wmf, \the\hsize #2}\fi
     \vspace{#2}
     \caption{#4}
     \mylabel{#3}
   \end{figure}
}

\def\FigureFromFileTwoD[#1][#2,#3](#4)#5
{
  \begin{figure}[htbp]
     \global\advance\figcount by 1
     \if\includefigures y\special{anisoscale #1.wmf, #2 #3}\fi
     \vspace{#2}
     \caption{#5}
     \mylabel{#4}
   \end{figure}
}

\def\FigureF<#1>[#2](#3)#4
{
  \begin{figure}[htbp]
     \global\advance\figcount by 1
     \if\includefigures y\special{anisoscale \figcat/fig\number\figcount.wmf,
       \the\hsize #2}
     \fi
     \if\includefigures p
       \leavevmode
       \epsfxsize=\hsize
       \epsffile{#1}
     \fi
     \if\includefigures y
          \vspace{#2}
     \fi
     \caption{#4}
     \mylabel{#3}
   \end{figure}
}

\def\Figure[#1](#2)#3
{
  \begin{figure}[htbp]
     \global\advance\figcount by 1
     \if\includefigures y\special{anisoscale \figcat/fig\number\figcount.wmf,
       \the\hsize #1}
     \fi
     \if\includefigures p
       \leavevmode
       \epsfxsize=\hsize
       \epsffile{fig\number\figcount.eps}
     \fi
     \if\includefigures y
          \vspace{#1}
     \fi
     \caption{#3}
     \mylabel{#2}
   \end{figure}
}

\newcommand{\kp}{\mathcal{K}}
\newcommand{\x}{\mathcal{X}}
\renewcommand{\pto}{\nrightarrow}
\renewcommand{\emptyset}{\varnothing}
\renewcommand{\subset}{\subseteq}

\renewcommand{\rho}{\varrho}
\renewcommand{\phi}{\varphi}
\renewcommand{\epsilon}{\varepsilon}

\newcommand{\Leray}{L_{\mathrm{Leray}}}

\def\Endo{\mbox{$\operatorname{Endo}$}}
\def\Auto{\mbox{$\operatorname{Auto}$}}
\def\Mono{\mbox{$\operatorname{Mono}$}}

\def\Path{\mbox{$\operatorname{Path}$}}
\newcommand{\lep}[2][]{#2^{\sqsubset#1}}
\renewcommand{\rep}[2][]{#2^{\sqsupset#1}}

\newcount\twoarrowsep
\makeatletter
\@namedef{dgo@dd}{\let\dg@VECTOR=\dg@twoarrowedvector}%

\newcount\dg@XSHIFT
\newcount\dg@YSHIFT
\def\dg@twoarrowedvector(#1,#2)#3{%
   \begingroup
   \dg@XTEMP=#1\relax\multiply\dg@XTEMP\m@ne\relax
   \dg@YTEMP=#2\relax\multiply\dg@YTEMP\m@ne\relax
   \dg@ZTEMP=#1\relax
   \ifnum\dg@ZTEMP<\z@
     \multiply\dg@ZTEMP\m@ne\relax \fi
   \ifnum\dg@YTEMP<\z@
     \advance\dg@ZTEMP by -\dg@YTEMP
   \else \advance\dg@ZTEMP by \dg@YTEMP \fi
   \dg@XSHIFT=#2\relax\multiply\dg@XSHIFT\m@ne\relax\multiply\dg@XSHIFT\twoarrowsep\relax
     \divide\dg@XSHIFT by \dg@ZTEMP\relax
   \dg@YSHIFT=#1\relax\multiply\dg@YSHIFT\twoarrowsep\relax\divide\dg@YSHIFT by \dg@ZTEMP\relax
   \begin{picture}(0,0)%
      \thinlines
       \put(0,10){\circle{0.1cm}}%
      \put(-\dg@XSHIFT,0){\line(#1,#2){#3}}%
    \end{picture}%
   \endgroup}%
\makeatother

\newcommand{\DS}{\displaystyle}

\articletheorems

\begin{document}

\author{Jonathan Barmak}
\address{Jonathan Barmak, Universidad de Buenos Aires,
Facultad de Ciencias Exactas y Naturales,
Departamento de Matem\'atica, Buenos Aires, Argentina.
CONICET-Universidad de Buenos Aires, Instituto de Investigaciones Matem\'aticas Luis A. Santal\'o (IMAS), Buenos Aires, Argentina}
\email{jbarmak@dm.uba.ar}

\author{Marian Mrozek}
\address{Marian Mrozek, Division of Computational Mathematics,
  Faculty of Mathematics and Computer Science,
  Jagiellonian University, ul.~St. \L{}ojasiewicza 6, 30-348~Krak\'ow, Poland}
\email{Marian.Mrozek@uj.edu.pl}
\author{Thomas Wanner}
\address{Thomas Wanner, Department of Mathematical Sciences,
George Mason University, Fairfax, VA 22030, USA}
\email{twanner@gmu.edu}
\date{today}
\thanks{J.B. is a researcher of CONICET; he is partially supported by grants PICT 2019-2338,
PICT-2017-2806, PIP 11220170100357CO, UBACyT 20020190100099BA, UBACyT 20020160100081BA.
The research of  M.M.\ was partially supported by
  the Polish National Science Center under Ma\-estro Grant No. 2014/14/A/ST1/00453 and Opus Grant No. 2019/35/B/ST1/00874.
  T.W.\ was partially supported by NSF grant DMS-1407087 and by the Simons Foundation
  under Award~581334.
}

\title[Conley index for multivalued maps on finite topological spaces]
{Conley index for multivalued maps on finite topological spaces}

\date{Version compiled on \today}

\begin{abstract}
We develop Conley's theory for multivalued maps on finite topological spaces.
More precisely, for discrete-time dynamical systems generated by the iteration
of a multivalued map which satisfies appropriate regularity conditions, we establish
the notions of isolated invariant sets and index pairs, and use them to introduce
a well-defined Conley index. In addition, we verify some of its fundamental
properties such as the Wa\.zewski property and continuation.
\end{abstract}

\maketitle



\section{Introduction}
\label{sec:intro}

Topological methods have always been at the heart of the qualitative study of
dynamical systems. For example, topological fixed point theorems can establish
the existence of stationary states based purely on topological properties of
the underlying system and the space it is acting on. But even more complicated
dynamical behavior can be studied in this way, for example recurrent and
chaotic dynamics. One of the central tools in this context was developed by
Charles Conley in~\cite{conley:78a}. He realized that rather than focusing on
the qualitative study of arbitrary invariant sets, it is advantageous to restrict
one's attention to isolated invariant sets. Broadly speaking, such sets are more
robust to continuous perturbations than general invariant sets. This insight allowed
Conley to associate an index to isolated invariant sets~$S$, which encodes some of
their dynamical properties. The Conley index of~$S$ can be determined without
explicit knowledge of the specific isolated invariant set through associated
index pairs, which provide rough topological enclosures of~$S$. In the case of classical
continuous-time dynamical systems the Conley index can either be defined
as a pointed topological space, or in a more computationally friendly version,
as a homology module. The Conley index for discrete-time dynamical systems, that is iterated maps on Hausdorff topological spaces, 
requires more advanced tools. 

In this paper we present Conley's theory for multivalued maps on finite topological spaces.
The need for such a theory is motivated by attempts to automate the computer aided analysis
of dynamical systems, particularly in modern data driven research.  
In the remainder of this section we explain this motivation in more detail, 
we review important earlier work on this problem, 
and we clarify the computational significance of the proposed theory.

\subsection{Motivation}
Despite the big success of Conley's theory in the qualitative study of differential equations, 
it was quickly realized that the manual verification of isolation is too laborious to be practical
for many concrete problems. 
This initiated interest in computer aided techniques and, 
since a digital computer is inherently finite,  led to questions about the discretization of the theory in time and in space. 
These questions are important and fundamental, because while algorithms act on finite structures, the outcome of the algorithms 
is interpreted in the domain of the analyzed problem --- which is typically a Euclidean space in the case of a differential equation. 
Therefore, a bridge is needed between the finite domain of algorithms and the uncountable domains of time and a Euclidean space. 
In addition, the more of the theory one can shift to the finite domain, the less human interaction is needed 
in guiding the algorithms. 

The discretization in time is achieved by means of numerical methods and leads to discrete-time dynamical systems, 
a subject interesting in itself, regardless of potential computer applications. 
A discretization in space can be achieved by putting a cellular structure on a compact subset of interest in the underlying Euclidean space. However, it is not clear what is the right notion of dynamical system on such a combinatorial object.

The early computer aided techniques focused on computer-assisted proofs. 
In such techniques one starts with knowledge about facts gathered from physical experiments,
and the ultimate goal is to confirm these facts by formal proofs with the assistance of a computer. 
This requires that the dynamical system is given by well-understood explicit formulas for a vector field or a homeomorphism.
Such a situation is typical in the physical sciences where very precise formulas may be obtained from the laws of physics.

The situation is different in biological and social sciences where  data collected from experiments, observations, or computations
is available, often in huge quantities, but with limited or no understanding of its meaning.
Therefore, the first goal is to extract some useful information from the data. 
And in such cases, even the phase space is known only via a sample, 
in the form of a point cloud. Contemporary topological data analysis (TDA, \cite{Ca2009,EH2010}) provides bridges
that enable the approximation of the phase space by a finite simplicial complex or a more general cell complex, whose dimension is usually much larger than that of the original ambient space.
This again leads to a discretization of space and facilitates the study of the topology of the phase space by algorithmic means.
Nevertheless, as we have already said, a useful dynamical system theory on such spaces is a challenge.
One option is to embed the cell complex in a Euclidean space
and interpolate the data in order to obtain an explicit formula for a vector field or a generator of a 
dynamical system. The problem is then reduced to a classical one, which typically 
is studied by numerical experiments. Unfortunately, this involves an artificial re-discretization with no direct 
connection to the original discrete data. 

Another option is to eliminate the classical dynamical system and to construct a combinatorial 
analogue of a dynamical system directly on the discretized space. A simplicial complex or a regular cell complex can be abstractly seen as a {\em Lefschetz complex}~\cite{Le1942}, also known as a {\em based chain complex}~\cite{BaRo2024},
or just as a finite topological space with Alexandrov topology~\cite{alexandrov:37a} induced via the face relation on the set of the cells. A dynamical system on a finite space has to be multivalued to be interesting.
In essence, such a system is a directed graph. However, it is not obvious 
how one can fruitfully use topology with such systems.
A hint comes from the seminal work of Robin Forman~\cite{forman:98a, forman:98b}
on a combinatorial version of Morse theory in the form of the concept of combinatorial vector field. 
While Forman's papers primarily served to extend Morse theory to the case of cell complexes, they also
addressed some more general dynamical concepts. Forman's ideas were then used
to construct Conley theory for combinatorial multivector fields on Lefschetz
complexes in~\cite{mrozek:17a}, and on general finite topological spaces
in~\cite{lipinski:etal:23a} (see also ~\cite{batko:etal:20a, mrozek:etal:22a, mrozek:wanner:21a}).

A combinatorial vector field is inherently a finite concept. Nevertheless, the combinatorial 
dynamical system it generates is a counterpart of a flow, that is, a dynamical system with continuous time.
This is reflected in a demanding definition of such a combinatorial system and,
in consequence, makes the construction of combinatorial vector fields a challenging task~\cite{Juda2020,WSLLK2023}.
In contrast, a multivalued map on a cellular space may easily be constructed from data by binning. 
Moreover, such a construction may be applied not only to data gathered from flows but also data gathered from dynamical systems with discrete time.
Our earlier work~\cite{barmak:etal:20a} indicates that iterated multivalued maps on finite topological spaces combine well with topological tools, providing
a natural combinatorial analogue of dynamical systems with discrete time. 
All this motivates our interest in Conley theory for such maps.

\subsection{Previous work} 
To the best of our knowledge, there is no previous work concerning Conley theory for multivalued maps on finite topological spaces.
Nevertheless, several earlier papers partially address issues which motivate our present work. 
One of the first such topics is the discretization in time, that is, a Conley theory for iterated maps.
The central challenge in this context is the lack of homotopies along trajectories for discrete time. 
In fact, entirely new techniques, based on so-called index maps, were needed. The first such technique, using shape theory, 
was proposed by Robbin and Salamon~\cite{RoSa1988}. A cohomological Conley index, 
more practical in computational terms, was proposed in \cite{mrozek:90a}.
A unifying, general approach was presented by Szymczak~\cite{szymczak:95a}. Later, Franks and Richeson~\cite{franks:richeson:00}
showed that Szymczak's approach is equivalent to a construction  using the more commonly known shift equivalence.

As we already mentioned, combinatorial dynamics has to be multivalued to be interesting.  
Conley's theory has successfully been extended to the case of multivalued discrete-time
dynamics, see for example~\cite{kaczynski:mrozek:95a, stolot:06a,
batko:mrozek:16a, batko:2017, batko:2023} and the references therein.
However, all of these papers concern only discretization in time, not in space. 

In the case of computer-assisted proofs, the need for a Conley theory in the discretized space may be avoided by using interval arithmetic
and multivalued dynamics. The outcome of computations may then be interpreted by means of the classical Conley theory. However, this requires
human support to guide the algorithms used in computer-assisted proofs. 
In consequence, such algorithms cannot be adapted to perform an automated search for interesting features. 
Yet, there is one notable exception: 
In the case of gradient dynamics,  Conley theory on the algorithmic level is not necessary
in order to detect Morse decompositions with the associated Conley-Morse graphs~\cite{BGHKMOP2012,BM2014,KMV:2014,KMV:2016,KMV:2022}.
All that is needed to guide the algorithms in this situation is the theory of directed graphs, combined with finite lattice theory.
No topology is necessary on the algorithmic side. This is related to the fact that a strongly connected 
component of a digraph is always an isolated invariant set for a combinatorial dynamical system --- regardless
of the used topology~\cite[Theorem 4.1]{dey:etal:19a}. However, topology matters tremendously in the case of recurrent dynamics~\cite{LMM2024}.

\subsection{Computational aspects}
As we already mentioned, algorithms addressing problems in continuous mathematics require bridges 
which tie the problems with the combinatorial data structures handled by the algorithm.
The better the bridges, the more automated the actions of the algorithms may be. 
In this paper, entirely motivated by computational issues, we do not present any algorithms. It is too early to do that. 
However, the paper provides the foundation on the combinatorial side for building a bridge towards classical Conley theory. 
Hence, it provides the first step needed to construct automated algorithms for the analysis of sampled dynamical systems. 

\subsection{Outline}
In the present paper, we aim to demonstrate that Conley's theory can be extended
to the case of general dynamical systems on finite topological spaces. As we will
see in more detail in Section~\ref{sec:comb-dyn}, actual dynamical systems on such
combinatorial objects necessarily have to be multivalued and time-discrete.
Thus, we consider the iteration of multivalued maps on finite topological spaces
and define the notions of isolated invariant sets and their Conley index. We
prove that the index is well-defined, and establish some of its basic properties.
While our approach is modeled after previous results~\cite{mrozek:90a,batko:mrozek:16a},
the involved proof techniques are significantly different. This is due to the lack of sufficient separation
in finite topological spaces, and will be addressed in more detail later.

The remainder of this paper is organized as follows. In Section~\ref{sec:prelim}
we recall basic definitions concerning finite topological spaces and continuity
properties of multivalued maps. This is followed in Section~\ref{sec:comb-dyn}
by a brief discussion of combinatorial topological dynamics, which specifically 
demonstrates that on finite topological spaces interesting dynamics can only be
observed in the context of iterating a multivalued map. In addition, we introduce
the central notion of solution in this context. We then turn our attention to
Conley theory. Section~\ref{sec:isoset} is devoted to isolated invariant sets
and Morse decompositions, while Section~\ref{sec:indexpair} is concerned
with index pairs and their properties. Using these results, we can define the
Conley index in Section~\ref{sec:conleyindex}, and derive some of its fundamental
properties in Section~\ref{sec:conleyindexprop}. Finally, Section~\ref{sec:future}
addresses some future work and open problems.


\section{Preliminaries}
\label{sec:prelim}

We begin by recalling basic concepts and definitions for finite topological spaces,
as well as for multivalued maps between them. While we focus only on the essentials,
additional material can be found in~\cite{alexandrov:37a, barmak:11a, barmak:etal:20a}.

By a finite topological space we mean a topological space with finitely many points. Given a finite topological space $X$ and a subspace $A$, we denote by $\opn A$ the open hull of $A$, that is, the smallest open set containing $A$. When $A$ consists of a unique point $a$ we also write $\opn A= \opn a$. Note that $\opn A= \bigcup\limits_{a\in A} \opn a$.
The closure of $A$ is denoted by $\cl A$. Notice that for arbitrary elements $x,y\in X$ the inclusion $x \in \opn y$ is satisfied if and only if $y\in \cl x$. Every finite space has an associated preorder~$\le$ (i.e., a reflexive and transitive relation) on its underlying set given by $x \le y$ if $x\in \cl y$.\footnote{Note that this convention is the one used in~\cite{alexandrov:37a}, and it is the most appropriate one for the setting of dynamics. We would like to point out, however, that alternatively the preorder could be defined by letting $x \le y$ if $x\in \opn y$.
This definition is also extensively used in the literature, see for example the discussion in~\cite{barmak:etal:20a}.} Conversely every finite set with a preorder $\le$ has a corresponding topology with the up-sets as the open sets.
Recall that a subset $A\subseteq X$ is an up-set, if $a\le x$ for some $a\in A$ implies $x\in A$.
Then, the dually defined down-sets correspond to the closed sets in this topology.
A finite space~$X$ is~$T_0$ (i.e., given any two different points there exists an open set containing only one of them) if and only if the preorder is an order (i.e., antisymmetric). Note that a finite space is Hausdorff (or even $T_1$) if and only if it is discrete. A map $f:X\to Y$ between finite spaces
is continuous if and only if it is order preserving, that is, if the inequality~$x\le x'$
always implies $f(x)\le f(x')$. 
Although this correspondence is very useful to understand finite spaces from a combinatorial perspective, we have chosen to use the topological notation $\cl A$ instead of $X_{\le A}=\{x\in X | \ \exists \ a \in A $ with $x\le a \}$ and $\opn A$ instead of $X_{\ge A}=\{x\in X | \ \exists \ a \in A $ with $a\le x \}$ in order to make more evident the connection between this theory and the classical one.

While it may not be evident for some readers, singular homology is well-defined for any topological space, in particular for a finite topological space~\cite{Munkres1984,Hatcher2002}.
For $T_0$ spaces it may be computed as the simplicial homology of the associated order complex.
Given a finite $T_0$ space $X$, its order complex $\kp(X)$ is an abstract simplicial complex whose simplices are the non-empty chains of the poset. Singular homology groups $H_n(X)$ of the finite space are isomorphic to the homology groups of the complex $\kp(X)$ (\cite{mccord:66a}). Therefore, to compute the homology of a finite $T_0$ space we can simply consider the free chain complex $C_*(X)$ were $C_n(X)$ is the free abelian group generated by the chains in $X$ of length $n$. In the other direction, any finite simplicial complex (or more generally any finite regular CW-complex) $K$ can be turned into a finite $T_0$ space by means of the face poset functor $\x$. The space $\x(K)$ is the poset of simplices of $K$ ordered by the face relation. In other words, its topology is given by $x\in \cl y$ if $x$ is a face of $y$. The (singular) homology groups of the finite space are isomorphic to those of the complex.

We say that a multivalued map $F: X \mto Y$ between two topological spaces has 
{\em closed values}, if $F(x)\subseteq Y$ is closed for every $x\in X$. Furthermore, the map~$F$
is called {\em lower semicontinuous} if the small preimage $F^{-1}(H)=\{x \in X | F(x)\subseteq H\}$
is closed for every closed subset $H\subseteq Y$. For a multivalued map $F:X\mto Y$ with
closed values between finite $T_0$ spaces, one can easily verify that being lower semicontinuous
is equivalent to the condition that $x' \le x$ implies $F(x') \subseteq F(x)$, or, in other
words, $x' \in \cl x$ implies $F(x') \subseteq F(x)$, see also~\cite[Lemma~3.5]{barmak:etal:20a}.
Finally, we say that~$F$ has {\em acyclic values}, if for every $x\in X$ the
subspace~$F(x) \subset Y$ is acyclic.

For the majority of the paper, we consider multivalued maps $F:X\mto Y$ between finite $T_0$ spaces which are lower
semicontinuous and have closed values. If we assume in addition that the map has acyclic
values, then the projection $p_1:F\to X$ from the graph
\begin{displaymath}
  F=\{(x,y)\in X\times Y \ | \ y\in F(x)\} \subseteq X\times Y
\end{displaymath}
into~$X$ induces isomorphisms in all homology groups. This in turn implies that for
such multivalued maps there is an induced homomorphism $F_*:H_*(X)\to H_*(Y)$ given
by $F_*=(p_2)_*(p_1)^{-1}_*$ where $p_2:F\to Y$ stands for the other projection (see \cite[Proposition 4.7]{barmak:etal:20a}).


\section{Combinatorial topological dynamics}
\label{sec:comb-dyn}

In this brief section we introduce the notion of a combinatorial dynamical system
on a finite topological space, as well as the assumed topological properties of its
multivalued generator~$F:X\mto X$. Moreover, we indicate why in the setting of a finite
topological space only discrete-time dynamics is of interest. We would like to point out,
however, that through the notion of combinatorial vector fields on finite topological
spaces, one can in fact arrive at a notion of dynamics which is similar in spirit to
the continuous-time case, albeit not the same. Finally, we introduce the
notion of solution, which is central to this paper.

\subsection{Multivalued dynamics on finite topological spaces}

Classical dynamical systems can broadly be divided into two categories --- discrete-time
and continuous-time dynamical systems. In the former case, one is interested in the evolution
of a system state at discrete points in time, and this is usually modeled by the iteration
of a continuous map~$F : X \to X$. Unfortunately, in the context of a finite topological space
this leads to trivial dynamical behavior, with every orbit of the system eventually becoming
periodic. Thus, in order to capture interesting dynamics, one is forced to consider multivalued
maps $F:X\mto X$. While this has already been described in~\cite{batko:etal:20a,
kaczynski:etal:16a, lipinski:etal:23a, mrozek:17a, mrozek:etal:22a, mrozek:wanner:21a},
these papers consider very specific multivalued maps generated by an underlying combinatorial 
vector field or combinatorial multivector field --- and this approach is more in the spirit
of the continuous-time case. See also our comments below.

In contrast, the present paper is devoted to the study of general multivalued discrete-time
dynamical systems on a finite topological space~$X$. Since such general systems cannot rely on
any supporting underlying structure such as a combinatorial multivector field, we need to
impose certain regularity assumptions on the map~$F$. Throughout this paper, we assume
that~$F:X\mto X$ is a lower semicontinuous multivalued map with closed values. These
assumptions are inspired by the case of classical multivalued dynamics~\cite{deimling:92a,
gorniewicz:06a}, and they have also been used recently in the proof of a Lefschetz fixed point
theorem for multivalued maps on finite spaces~\cite{barmak:etal:20a}. We think of the map~$F$
as a {\em combinatorial dynamical system\/}, which is obtained by iterations of the map, and
which naturally leads to the concept of a {\em solution\/} --- as described in more detail in
the following section. For now we would like to point out that a combinatorial dynamical system
may also be viewed as a finite directed graph whose set of vertices is the topological space~$X$,
and with~$F$ interpreted as the map sending a vertex to the collection of its neighbors connected
via an outgoing directed edge. This so-called {\em $F$-digraph} encodes the dynamics of~$F$
on a purely combinatorial level. However, for the derivation of more advanced concepts such
as isolated invariant sets and their Conley index the topological properties of~$X$ and~$F$
are essential. 

In view of our focus on the discrete-time case, it is natural to wonder why we exclude the 
continuous-time case. As the following result shows, the semigroup property of a multivalued
continuous-time dynamical system immediately forces the dynamics to be trivial. In fact,
every orbit of the system has to be constant.
\begin{thm}[Triviality of continuous-time dynamics]
Let $X$ be a finite set and let $F: X \times \RR_{\ge 0} \multimap X$ denote a multivalued
map which satisfies the semigroup property $F(x,t+s)=F(F(x,t),s)$ for every $t,s\ge 0$.
Then $F(x,-): \RR_{> 0}\multimap X$, given by $t\mapsto F(x,t)$, is constant for
every $x\in X$.
\end{thm}
\begin{proof}
The map $F$ induces a singlevalued map $F:\mathcal{P}(X)\times \RR_{\ge 0} \to \mathcal{P}(X)$ given by $(A,t)\mapsto F(A,t)$. Here $\mathcal{P}(X)$ denotes the power set of $X$. Note that the identity  $F(A,t+s)=F(F(A,t),s)$ holds for every $t,s\ge 0$. Since $\mathcal{P}(X)$ is finite, it suffices to prove the following assertion:
\begin{itemize}
\item If $Y$ is a finite set and $F:Y\times \RR_{\ge 0} \to Y$ is a singlevalued map satisfying the semigroup property $F(y,t+s)=F(F(y,t),s)$ for every $t,s\ge 0$, then the map $F(y,-)$ is constant on~$\RR_{>0}$ for each $y\in Y$.
\end{itemize}
To show this, note that if $f:Y \to Y$ is any map, then the sequence $(f^n(Y))_{n\in \NN}$ is decreasing. We call $f^{\infty}(Y)\subseteq Y$ its eventual value. It is clear that $f^{\infty}(Y)=f^n(Y)$ for every~$n$ greater than or equal to the cardinality~$N$ of~$Y$. The map $f$ induces a bijection from the eventual value $f^{\infty}(Y)$ to itself, and since the group of bijections has order dividing $N!$, the map $f^{N!}:f^{\infty}(Y) \to f^{\infty}(Y)$ is the identity. This in turn shows that $f^{N!}:Y \to f^{\infty}(Y)$ is a retraction, i.e., it is the identity when restricted to its codomain.

Every $t\ge 0$ induces a map $F_t:Y \to Y$, $y\to F(y,t)$. Denote $R_t=F_t^{\infty}(Y) \subseteq Y$. By the comments above, the iterate $F_t^{N!}: Y \to R_t$ is a retraction. Furthermore, the set~$R_t$ is the set of fixed points of~$F_t^{N!}$. Now let $n\in \NN$ and $t\ge 0$. Then we have $F_{nt}=F_t^n$ in view of our hypothesis on~$F$. Thus $F_{nt}^{N!}=F_t^{nN!}$ fixes every point of $R_t$ and does not fix any point outside $R_t$. This proves that $R_{nt}=R_t$. We deduce that for $t>0$, the set~$R_t$ depends only on the class of~$t$ modulo~$\QQ_{>0}$, i.e., we have $R_t=R_{s}$ if $t^{-1}s\in \QQ$. In particular, this implies $R_{t/N!}=R_t$, and thus~$R_t$ is the image of~$F_t$, and~$F_t$ is the identity on~$R_t$.

Finally, let $t,s > 0$. Since $F_{t+s}=F_sF_t$, the image of $F_{t+s}$ is contained in the image of $F_s$, i.e., $R_{t+s}\subseteq R_s$. Thus $R_t\subseteq R_s$ for every $t\ge s$. Since $R_s=R_{s/n}$ for every $n\in \NN$, one also obtains $R_t \subseteq R_s$ for every $t>0$. This in turn establishes the identity $R_t=R_s$ for every $t,s>0$. Suppose $s>t>0$. Then $F_{s-t}$ is the identity on $R_{s-t}=R_s=R_t$. Since $F_{s}=F_{s-t}F_t$ and $F_{s-t}$ is the identity on the image of $F_t$, then $F_s=F_t$. This proves the assertion.
\qed
\end{proof}

\medskip
The above result shows that it is the semigroup property alone which is incompatible
with nonconstant dynamics if the underlying phase space is finite. As the reader 
undoubtedly noticed, we did not make use of any topological structure on~$X$.
Note, however, that one can mimic the behavior of a continuous-time dynamical
system even on finite topological spaces by restricting dynamical transitions
between subsets to shared boundaries. This is precisely what Forman had in mind
with his combinatorial vector fields, and also lies at the center of the theory
of multivector fields. In contrast, the discrete-time dynamics studied in the
present paper does not have these restrictions, as it allows for transitions
between states without topological closeness.

\subsection{Solutions and invariant sets}

Our study of the dynamics of discrete-time multivalued dynamical systems is based
on the notions of solution and invariant set. These are defined just as in the 
classical situation.

Consider a multivalued map $F: X \mto X$. Then a {\em solution} of $F$ in $A\subset X$
is a partial map $\sigma:\ZZ\pto A$ whose {\em domain}, denoted $\dom \sigma$, is an
interval of integers, and for any $i,i+1\in\dom\sigma$ the inclusion $\sigma(i+1)\in
F(\sigma(i))$ is satisfied. The solution~$\sigma$ is called a {\em full solution\/}
if $\dom \sigma=\ZZ$, otherwise it is a {\em partial solution}. A partial solution
whose domain is bounded is referred to as a {\em path}. We denote the set of all
paths with values in $A\subset X$ by $\Path(A)$. Given a path~$\sigma$ with domain
$\dom\sigma=\ZZ\cap [m,n]$ for some $m,n \in \ZZ$, we call~$\sigma(m)$ and~$\sigma(n)$,
respectively, the left and right {\em endpoint} of~$\sigma$. We denote these endpoints
by the symbols~$\lep{\sigma}$ and~$\rep{\sigma}$, respectively. 


If~$\tau$ is another path with $\dom\tau=\ZZ\cap [m',n']$ and such that $\lep{\tau}
\in F(\rep{\sigma})$ holds, then we define the concatenation of the paths~$\sigma$
and~$\tau$, denoted by~$\sigma.\tau$, as the path with domain $\dom \sigma.\tau :=
\ZZ\cap [m,n+n'-m'+1]$ and defined by
\[
  (\sigma.\tau)(k):=\begin{cases}
                      \sigma(k) & \text{ if $k\in \ZZ\cap [m,n]$},\\
                      \tau(k+m'-n-1) & \text{ if $k\in \ZZ\cap [n+1,n+1+n'-m']$}.
                    \end{cases}
\]
It is straightforward to verify that~$\sigma.\tau$ is indeed a path. 

We now recall the definition of invariance. For this, we say that a solution~$\sigma$
{\em passes} through $x\in X$ if $x=\sigma(i)$ for some $i\in\dom\sigma$. Moreover,
a set~$A\subset X$ is called {\em invariant} if for every $x\in A$ there exists a
full solution in~$A$ which passes through~$x$.
Thus, $A$ is invariant if $A\subset F(A)$ and for each $a\in A$, $F(a)\cap A\neq \emptyset$.


\section{Isolated invariant sets and Morse decompositions}
\label{sec:isoset}

The concept of isolated invariant set lies at the heart of Conley theory. In the
classical situation, an isolated invariant set~$S$ is characterized by the property
that it is the largest invariant set in some neighborhood of~$S$. Unfortunately,
it is not possible to define isolated invariant sets in an analogous way in the
context of finite topological spaces due to the lack of sufficient separation. Therefore, in
this section we introduce an appropriate notion for our setting and derive some
first properties of such isolated invariant sets. We also show how they form the
building blocks for Morse decompositions of phase space. Throughout this section,
we assume that~$X$ is a finite~$T_0$ topological space and that the multivalued
map $F:X\mto X$ is lower semicontinuous with closed values.

\subsection{Isolated invariant sets}
\label{ssec:iso-inv-set}

We begin by introducing the notion of isolating invariant set, which in turn is
based on an isolating set. The latter set is the analogue of the isolating neighborhood
in classical Conley theory, but its topological properties are weaker to account for the
poor separation in finite spaces.

\begin{defn}[Isolated invariant set, isolating set]
\label{defn:isolating-set}
A closed set $N\subset X$ is called an {\em isolating set\/} for an invariant set
$S$ if the following two conditions are satisfied:
\begin{itemize} \setlength{\itemsep}{3pt}
   \item[(IS1)] Every path in~$N$ with endpoints in~$S$ has all its values in~$S$.
   \item[(IS2)] We have the equality $S\cap\cl(F(S)\setminus N) = \emptyset$, i.e., the
                set~$S$ and~$\cl(F(S)\setminus N)$ are disjoint.
\end{itemize}
If such an isolating set for~$S$ exists, we say that~$S$
is an {\em isolated invariant set}.
\end{defn}

Notice that condition~(IS2) is satisfied
if and only if $\opn S \cap (F(S)\setminus N)=\emptyset$. Thus, it is equivalent to 
assuming the inclusion
\begin{itemize} \setlength{\itemsep}{3pt}
   \item[(IS2')] $\opn S \cap F(S) \subset N$.
\end{itemize}
Note also that since $S$ is invariant, $S\subset F(S)$, and hence (IS2') implies $S\subset N$.

Establishing condition~(IS2), or its equivalent reformulation~(IS2'), is the less
intuitive aspect of verifying an invariant set as an isolated invariant set. It is
therefore useful to also have sufficient conditions for its validity. Two of these
are the subject of the following remark.
\begin{rem}[Sufficient conditions for~(IS2)]
\label{rem:suffcond-is2}
Assume that~$S$ is an invariant set and that $N$ is closed. Then any of the
following two conditions imply~(IS2):
\begin{enumerate}
 \item[(i)]  We have $S \subset \inte N$, where~$\inte N$ denotes the interior of $N$, or
 \item[(ii)] the inclusion $F(S) \subset  N$ is satisfied.
\end{enumerate}
Indeed, the first condition is equivalent to $\opn S \subseteq N$, and therefore either
of the above two conditions implies~(IS2'), and thus~(IS2).
\end{rem}
As we mentioned earlier, in classical Conley theory, the isolated invariant
set is uniquely determined by its isolating neighborhood~$N$. In fact, it is the
largest invariant subset of~$N$. In contrast, in the above setting the same set~$N$
may be an isolating set for more than one isolated invariant set. This is illustrated
in the following two examples.
\begin{ex}[A rotational multivalued map]
\label{ex:simplecvf}
{\em 
We begin with a simple example that rotates an equilateral triangle. In the left
part of Figure~\ref{fig:simplecvf} we indicate the action of the map on a
simplicial complex, which is just a two-dimensional simplex. More precisely, the map rotates the triangle in a counterclockwise fashion by~120$^\circ$. This example is inspired by a combinatorial vector field
in the sense of Forman, which contains the three vectors $\{ A, AB \}$, $\{ B, BC \}$,
and $\{ C, AC \}$ along the boundary, as well as the critical cell $\{ ABC \}$.
While we refer the reader to~\cite{forman:98a, forman:98b, kaczynski:etal:16a,
mrozek:wanner:21a} for more details on the general definition of a combinatorial
vector field and its relation to classical dynamics, it is intuitively clear that
in the situation of Figure~\ref{fig:simplecvf} one can observe both an unstable
fixed point at the triangle, as well as periodic motion along its simplicial boundary.
\begin{figure} \centering
  \setlength{\unitlength}{1 cm}
  \begin{picture}(12.0,4.0)
    \put(0.0,0.0){
      \includegraphics[height=4cm]{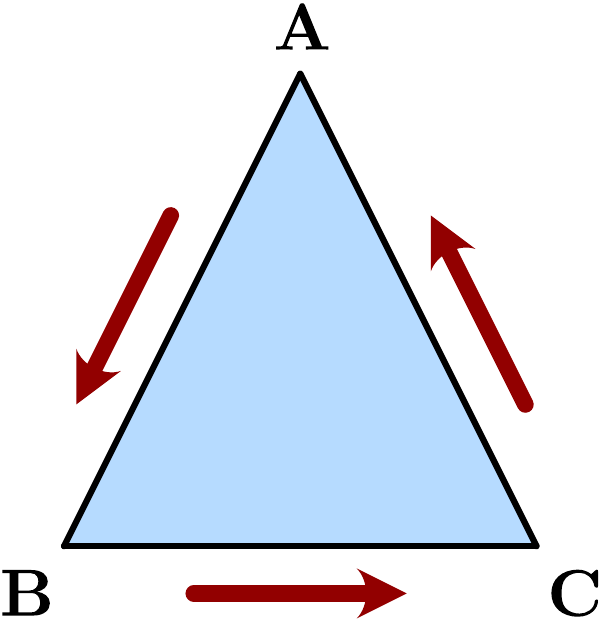}}  
    \put(7.0,0.0){\makebox(4.0,4.0){
      \begin{tabular}{|c||c|} \hline
        $x$ & Elements of $F(x)$ \\[0.5ex] \hline
        A   & B \\
        B   & C \\
        C   & A \\
        AB  & B, C, BC \\
        BC  & A, C, AC \\
        AC  & A, B, AB \\
        ABC & A, B, C, AB, BC, AC, ABC \\ \hline
      \end{tabular}
      }} 
  \end{picture}
  \caption{A simple rotation on a simplicial complex given by one
           triangle, as well as three edges and three vertices. The table
           on the right defines the associated multivalued map $F : X \mto X$ on the finite topological space consisting of all seven simplices, and equipped with the closure operation
           induced by the face relationship.}
  \label{fig:simplecvf}
\end{figure}%

In order to formulate this dynamical behavior via a multivalued map on a finite
topological space, we use the standard construction given by the face poset, that is, the poset $X$ of simplices where $x\le y$ if $x$ is a face of $y$. As we have already mentioned, this corresponds to the topology given by $x \in \cl y$ if and only if~$x$ is a face of~$y$. The associated
multivalued map $F : X \mto X$ is defined in the table in Figure~\ref{fig:simplecvf}.
One can easily verify that~$F$ has closed values, and that it is lower semicontinuous.
Iteration of the map~$F$ leads for example to the following three isolated invariant
sets:
\begin{displaymath}
  S_1 = \{ A, B, C \} \; , \quad
  S_2 = \{ AB, BC, AC \} \; , \quad\mbox{ and }\quad
  S_3 = \{ ABC \} \; .
\end{displaymath}
If we then define the closed sets
\begin{displaymath}
  N_1 = \{ A, B, C \} \; , \quad
  N_2 = \{ A, B, C, AB, BC, AC \} \; , \quad\mbox{ and }\quad
  N_3 = X \; ,
\end{displaymath}
then one can easily verify that~$N_3$ is an isolating set for all three of the
above isolated invariant sets, the set~$N_2$ isolates both~$S_1$ and~$S_2$, and
the set~$N_1$ is an isolating set for~$S_1$ only. Finally, we note that also 
the unions~$S_1 \cup S_2$ and~$S_2 \cup S_3$ are isolated invariant sets,
with isolating sets~$N_2$ and~$N_3$, respectively.

On the other hand, while
the union~$S_1 \cup S_3$ is invariant, it is not an isolated invariant set.
To see this, note that any isolating set for~$S_1 \cup S_3$ has to contain
the closure of~$S_1 \cup S_3$, and therefore $N = X$ would be the only
possibility. Yet, one can easily see that~(IS1) is not satisfied for
this choice.
}
\end{ex}
\begin{ex}[A reflection-based multivalued map]
\label{ex:reflectedcvf}
{\em
Our second example is similar to the previous one but it is induced by the reflection of the triangle about the vertical line through~$A$,
as depicted in the left panel of Figure~\ref{fig:reflectedcvf}. The
corresponding multivalued map $G : X \mto X$ is defined in the table on the
right. Notice that also~$G$ has closed values and is lower semicontinuous.
\begin{figure} \centering
  \setlength{\unitlength}{1 cm}
  \begin{picture}(12.0,4.0)
    \put(0.0,0.0){
      \includegraphics[height=4cm]{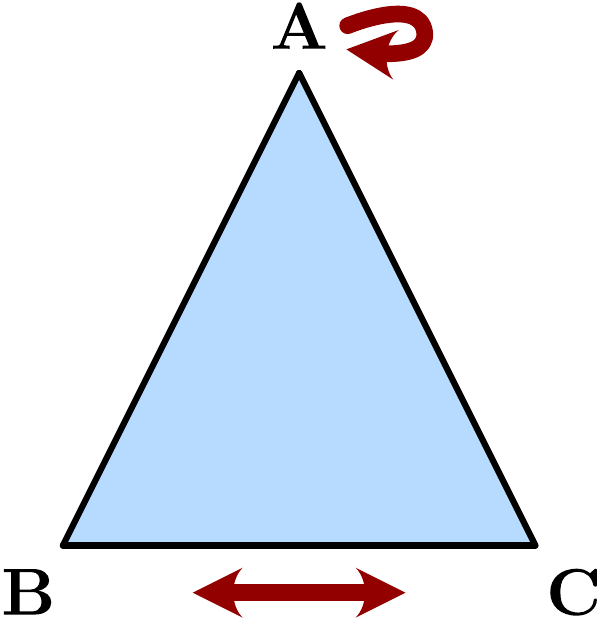}}  
    \put(7.0,0.0){\makebox(4.0,4.0){
      \begin{tabular}{|c||c|} \hline
        $x$ & Elements of $G(x)$ \\[0.5ex] \hline
        A   & A \\
        B   & C \\
        C   & B \\
        AB  & A, C, AC \\
        BC  & B, C, BC \\
        AC  & A, B, AB \\
        ABC & A, B, C, AB, BC, AC, ABC \\ \hline
      \end{tabular}
      }} 
  \end{picture}
  \caption{The map $G:X \mto X$ defined on the right induced
           by a reflection about the vertical line through $A$ indicated on the left.}
  \label{fig:reflectedcvf}
\end{figure}%

Iteration of the map~$G$ leads to new isolated invariant sets. For example,
both the singleton $R_1 = \{ A \}$ and the doubleton $R_2 = \{ B, C \}$ are
examples, and they have associated isolating sets $M_1 = R_1$ and $M_2 = R_2$,
respectively. Notice, however, that both sets are also isolated by $M = X$.
In addition, we have the isolated invariant sets
\begin{displaymath}
  R_3 = \{ BC \} \; , \quad
  R_4 = \{ AB, AC \} \; , \quad\mbox{ and }\quad
  R_5 = \{ ABC \} \; .
\end{displaymath}
If we then define the closed sets
\begin{displaymath}
  M_3 = \{ B, C, BC \} \; , \quad
  M_4 = \{ A, B, C, AB, AC \} \; , \quad\mbox{ and }\quad
  M_5 = X \; ,
\end{displaymath}
then one can easily verify that~$M_k$ is an isolating set for~$R_k$ for $k = 3,4,5$.
Furthermore, the set~$M_5$ isolates both~$R_3$ and~$R_4$ as well. We leave it to
the reader to find additional isolated invariant sets.
}
\end{ex}
The examples above will be analyzed along the paper. We have chosen finite spaces associated with simplicial complexes because the geometric interpretation they have make notions simpler to visualize. However, we want to stress that the theory we develop here can be applied to any finite $T_0$ space.

While at first glance the nonuniqueness of the isolating set seems strange,
it is necessary in finite topological space due to the lack of sufficient
separation. Nevertheless, the following remark sheds more light on this issue.
\begin{rem}[The smallest isolating set]
\label{newr}
It is clear that there is a smallest closed set~$N$ satisfying condition~(IS2'), which is the set
\begin{equation} \label{newr-eq1}
  N = \cl (\opn S \cap F(S)) \; .
\end{equation}
On the other hand, one can easily see that condition~(IS1) is preserved by taking
subsets: If $N'\subset N$ and~$N$ satisfies this condition, then so does~$N'$. In
conclusion, the invariant set~$S$ is an isolated invariant set if and only if the
set~$N$ defined in~(\ref{newr-eq1}) satisfies condition~(IS1). We will see, however,
that it is frequently useful to work with different isolating sets for the same
isolated invariant set.
\end{rem} 
We leave it to the reader to illustrate the above remark in the context of
Examples~\ref{ex:simplecvf} and~\ref{ex:reflectedcvf}, and close this section
with the following simple result.
\begin{prop}
\label{prop:iso-set-intersection}
Assume~$M$ and~$N$ are two isolating sets for an isolated invariant set~$S$. 
Then their intersection~$M\cap N$ is also an isolating set for~$S$.
\end{prop}
\begin{proof}
Clearly the intersection $M\cap N$ is closed. Since every path in $M\cap N$
is also a path in~$M$, property~(IS1) for $M\cap N$ follows from the validity
of property~(IS1) for~$M$. Finally, it is clear that~(IS2') for~$M$ and~$N$
implies that~(IS2') also holds for $M\cap N$.
\qed
\end{proof}
%
\subsection{Morse decompositions}

Isolated invariant sets as defined in the last section are the fundamental
building blocks for analyzing the global dynamics of a dynamical system.
In general, they can be used to divide phase space into regions of 
recurrent and gradient-like behavior. This leads to the notion of a 
Morse decomposition.

\begin{defn}[Morse decomposition]
\label{defn:morese-decomp}
Consider a lower semicontinuous multivalued map $F:X\mto X$ with closed values,
on a finite~$T_0$ topological space~$X$.
A family $\{M_p\}_{p\in P}$ of mutually disjoint, non-empty, isolated invariant sets
indexed by a poset $P$ is called a {\em Morse decomposition} of~$X$ if for every full
solution~$\gamma$ either all values of~$\gamma$ are contained in the same set~$M_p$,
or there exist indices $q>r$ in~$P$ and $t_q,t_r\in\ZZ$ such that $\gamma(t)\in M_q$
for $t\leq t_q$ and $\gamma(t)\in M_r$ for $t\geq t_r$. In the latter case, the
solution~$\gamma$ is called a {\em connection} from~$M_q$ to~$M_r$. Furthermore,
the sets~$M_p$ are called the {\em Morse sets} of the Morse decomposition.
\end{defn}
In the context of classical dynamics, Morse decompositions are a fairly
difficult object of study, since it is possible for a dynamical system
to have infinitely many different Morse decompositions. Of course, this
cannot happen in the setting of a finite topological space. In fact, there 
is always a finest Morse decomposition which can easily be determined using
graph theoretic methods.

To see this, recall that the dynamics of a multivalued map $F:X\mto X$ can
be encoded via its $F$-digraph~$G_F$, whose vertices are given by the
elements of~$X$, and such that there is a directed edge from~$x$ to~$y$
if and only if~$y \in F(x)$. On~$X$, we can define an equivalence relation by
saying that~$x \sim y$ if and only if there is both a directed path in~$G_F$ 
from~$x$ to~$y$, and one from~$y$ to~$x$.\footnote{Notice that we have~$x \sim x$
for every~$x \in X$, since there always exists a path of length zero from~$x$ to~$x$,
i.e., a path without edges.} This equivalence relation partitions~$X$ into
equivalence classes which are called the {\em strongly connected components\/}
of~$G_F$. Such a component is called {\em trivial\/}, if it consists of a single
vertex which is not connected to itself with an edge, otherwise it is
{\em non-trivial\/}. Moreover, if each strongly connected component (along with all the edges which begin and finish in the component) is
contracted to a single vertex, the resulting graph is a directed acyclic
graph, called the {\em condensation\/} of~$G_F$. After these preparations,
one obtains the following result.
\begin{prop}[Morse decomposition via strongly connected components]
\label{prop:scc-as-morsedecomp}
Consider a lower semicontinuous multivalued map $F:X\mto X$ with closed values
on a finite~$T_0$ topological space~$X$. Denote the non-trivial strongly
connected components of the associated $F$-digraph~$G_F$ by $\{M_p\}_{p\in P}$.
Furthermore, let $q > r$ if there exists a directed path in~$G_F$ from~$M_q$ 
to~$M_r$. Then each of the sets~$M_p$ is an isolated invariant set for~$F$,
and $\{M_p\}_{p\in P}$ is a Morse decomposition of~$X$.
\end{prop}
\begin{proof}
We begin by showing that any path which starts and ends in~$M_p$ has to be
completely contained in~$M_p$.

To see this, let $p \in P$ be fixed, let $x,y \in M_p$, let~$\gamma$ denote
any path from~$x$ to~$y$, and let~$z$ denote any point on the path~$\gamma$.
Then~$\gamma$ clearly can be restricted to a path from~$z$ to~$y$. Furthermore,
since~$M_p$ is a non-trivial strongly connected component of~$G_F$, there
exists a path from~$y$ to~$x$. Concatenation of this path with the part
of~$\gamma$ from~$x$ to~$z$ gives a path in~$G_F$ from~$y$ to~$z$. This
immediately implies that~$y \sim z$, and therefore we have $z \in M_p$,
and the above statement follows.

We now turn to the verification of the proposition. It is easy to see
that~$>$ is indeed a (strict) partial order on~$P$, since the condensation of~$G_F$
is acyclic. Moreover, for any $x \in M_p$ one can easily construct a full
solution through~$x$ in~$M_p$, by infinite concatenations of the paths
from~$x$ to~$y$ and from~$y$ to~$x$, for some~$y \in M_p$, again using the
above observation. Thus, every set~$M_p$ is invariant. These sets are
also isolated invariant sets, since the whole space~$X$ is an isolating
set for each~$M_p$. For this, note that~(IS2) follows trivially from
Remark~\ref{rem:suffcond-is2}, and~(IS1) from our above observation.
Finally, if~$\gamma$ denotes an arbitrary full solution, then the acyclicity
of the condensation of~$G_F$ together with the finiteness of~$X$ immediately
implies the existence of $t_q,t_r\in\ZZ$ such that $\gamma(t) \in M_q$
for $t\leq t_q$ and $\gamma(t)\in M_r$ for $t\geq t_r$, for some $q \ge r$.
If $q = r$, then our above observation implies that~$\gamma$ is contained
in~$M_q$, and this completes the proof of the proposition.
\qed
\end{proof}

\medskip
Finding strongly connected components in digraphs can be done efficiently,
and thus the problem of decomposing the dynamics of a multi-valued map~$F$
into {\em recurrent dynamics\/}, given by the Morse sets~$M_p$, and
{\em gradient-like dynamics\/}, encoded in the condensation of~$G_F$,
is inherently computable. Furthermore, one can easily see that the above
result does in fact produce the finest Morse decomposition of~$X$.
It is customary to represent the information about this Morse decomposition
in the form of its {\em Morse graph}. This graph consists of the Hasse
diagram of the poset~$P$ with vertices representing the individual
Morse sets~$M_p$. In other words, it is the subgraph of the condensation
induced by the non-trivial strongly connected components.
\begin{ex}[Morse decompositions for Examples~\ref{ex:simplecvf} and~\ref{ex:reflectedcvf}]
\label{ex:morsedecomp}
{\em
We return to the two examples introduced earlier in this section. Recall
that these examples introduced two multivalued maps $F,G : X \mto X$
on the finite topological space
\begin{displaymath}
  X = \{ A, B, C, AB, AC, BC, ABC \}
\end{displaymath}
induced by a two-dimensional simplex. In Example~\ref{ex:simplecvf} we
identified the three isolated invariant sets
\begin{displaymath}
  S_1 = \{ A, B, C \} , \quad
  S_2 = \{ AB, BC, AC \} , \quad\mbox{ and }\quad
  S_3 = \{ ABC \} ,
\end{displaymath}
and one can easily see that they are all strongly connected components
of the $F$-digraph. Similarly, in Example~\ref{ex:reflectedcvf} we found
the isolated invariant sets
\begin{displaymath}
  R_1 = \{ A \} , \;\;
  R_2 = \{ B, C \} , \;\;
  R_3 = \{ BC \} , \;\;
  R_4 = \{ AB, AC \} , \;\;
  R_5 = \{ ABC \} ,
\end{displaymath}
which also partition the space~$X$ and are again strongly connected
components. Thus, in both of these examples all strongly connected
components are non-trivial, and one obtains the Morse graphs shown
in Figures~\ref{fig:simplemorse} and~\ref{fig:reflectedmorse},
respectively.
\begin{figure} \centering
  \setlength{\unitlength}{1 cm}
  \begin{picture}(8.0,4.0)
    \put(0.0,0.0){
      \includegraphics[height=4cm]{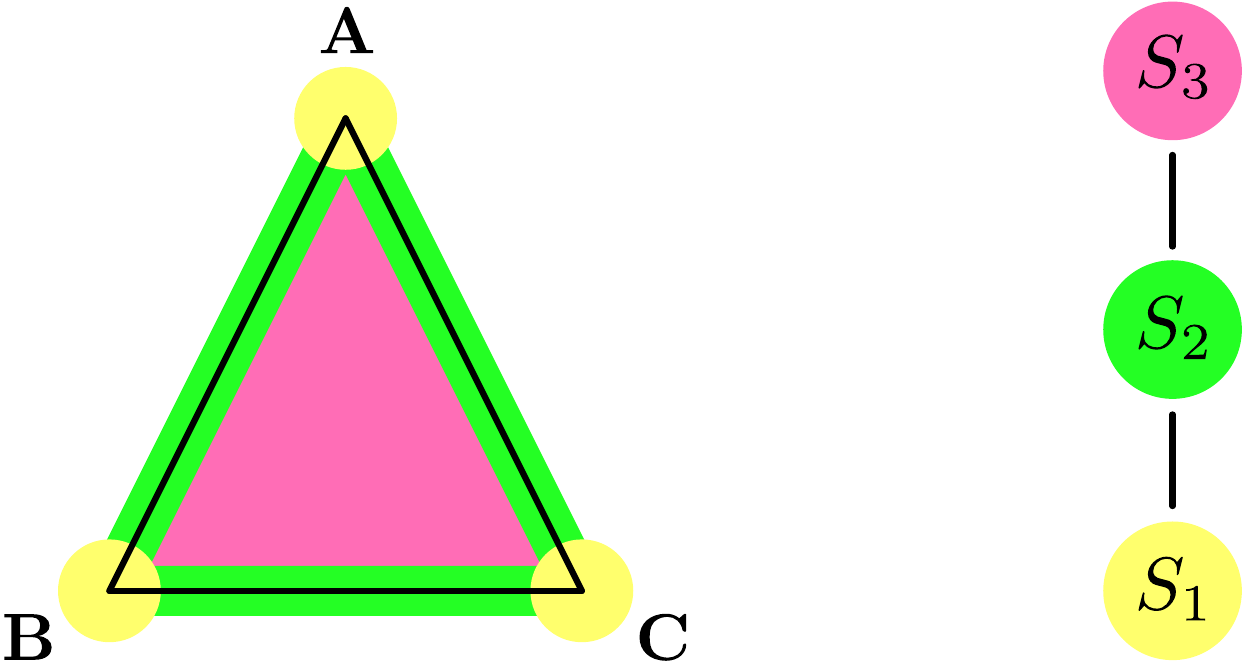}}
  \end{picture}
  \caption{Finest Morse decomposition for the map $F : X \mto X$
           from Example~\ref{ex:simplecvf}. The Morse graph is shown
           on the right, the Morse sets are indicated on the left.}
  \label{fig:simplemorse}
\end{figure}%
\begin{figure} \centering
  \setlength{\unitlength}{1 cm}
  \begin{picture}(8.0,4.0)
    \put(0.0,0.0){
      \includegraphics[height=4cm]{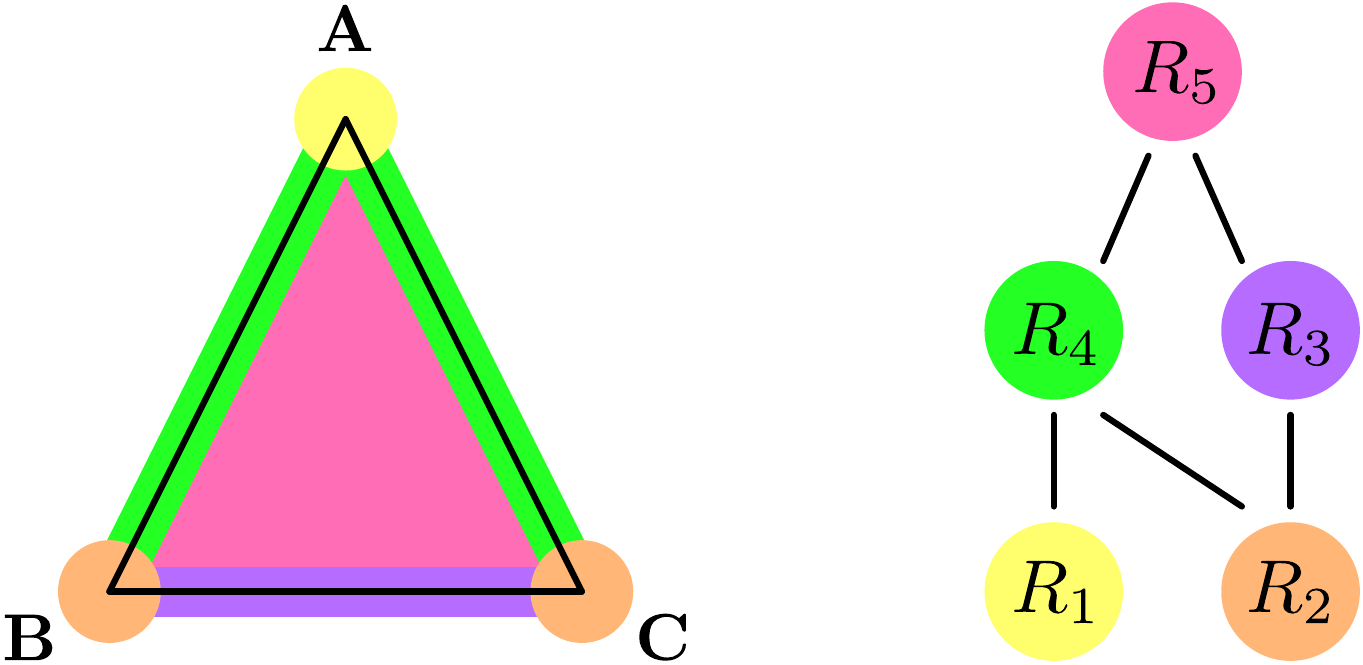}}
  \end{picture}
  \caption{Finest Morse decomposition for the map $G : X \mto X$
           from Example~\ref{ex:reflectedcvf}. The Morse graph is shown
           on the right, the Morse sets are indicated on the left.}
  \label{fig:reflectedmorse}
\end{figure}%

We would like to point out that the Morse sets involved in these 
example exhibit different types of recurrent dynamics. While the
sets~$S_3$, $R_1$, $R_3$, and $R_5$ are all equilibria of the dynamics,
the remaining Morse sets are periodic orbits. More precisely, the
sets~$S_1$ and~$S_2$ are periodic orbits of period~$3$, while the
two sets~$R_2$ and~$R_2$ have period~$2$.
}
\end{ex}
%
%
%
\section{Index pairs}
\label{sec:indexpair}
While isolated invariant sets~$S$ are the fundamental objects of study in
Conley theory, it is their Conley index that provides algebraic information
about the dynamics inside of~$S$. In classical dynamics, this index
information can be computed easily from certain isolating neighborhoods
called isolating blocks, but these are often difficult to find. For this
reason, one usually uses a different object for the index computation,
called an {\em index pair\/}. In the present section, we transfer this
concept to the setting of multivalued maps. Throughout, we again assume
that~$X$ is a finite~$T_0$ topological space and that the multivalued
map $F:X\mto X$ is lower semicontinuous with closed values.
%
%
%
\subsection{Definition and existence of index pairs}
\label{sec:indexpairdef}
For the definition of an index pair, we need to recall two
important concepts. On the one hand, if $A\subseteq X$ is any subset, then
the {\em invariant part\/} $\Inv (A)$ of~$F$ in~$A$ is the set of all
points $x\in A$ for which there exists a full solution in~$A$ which passes
through~$x$. On the other hand, a {\em topological pair\/} in~$X$ is a
pair~$P$ of subsets~$P=(P_1,P_2)$ which satisfy the inclusion
$P_2 \subset P_1$. With this, we have the following central definition.
\begin{defn}[Index pair]
%
%
We say that a topological pair $P=(P_1,P_2)$ of closed subsets of an
isolating set~$N$ for an isolated invariant set~$S$ is an {\em index
pair\/} for~$S$ in~$N$ if the following three conditions are satisfied:
\begin{itemize} \setlength{\itemsep}{3pt}
\item[(IP1)] $F(P_i)\cap N\subset P_i$ for $i=1,2$,
\item[(IP2)] $P_1\cap\cl(F(P_1)\setminus N)\subset P_2$,
\item[(IP3)] $S=\Inv(P_1\setminus P_2)$.
\end{itemize}
In addition, we say that index pair $P=(P_1,P_2)$ is {\em saturated\/}
if $S=P_1\setminus P_2$.
%
%
\end{defn}
We would like to point out that condition~(IP1) implies $F(P_i) \cap N =
F(P_i) \cap P_i$, and therefore~(IP2) could also be replaced by the
inclusion $P_1\cap \cl (F(P_1)\setminus P_1)\subset P_2$ to obtain an
equivalent definition.  

In the remainder of this section, we establish some basic properties
of index pairs. In addition, we show that every isolated invariant set~$S$
with isolating set~$N$ does indeed have an associated index pair. For this
we need another definition. For subsets $S\subset N\subset X$ we define
\begin{eqnarray*}
  \Inv^-(N,S)&:=&\setof{y\in N\mid \exists \sigma\in\Path(N)\; 
                 \text{ with }\; \lep{\sigma}\in S, \; \rep{\sigma}=y},\\[0.5ex]
  \Inv^+(N,S)&:=&\setof{y\in N\mid \exists \sigma\in\Path(N)\;
                 \text{ with }\; \lep{\sigma}=y, \; \rep{\sigma}\in S}.
\end{eqnarray*}
In other words, the set~$\Inv^+(N,S)$ consists of all points in~$N$
from which one can reach~$S$ in forward time with a path in~$N$,
and~$\Inv^-(N,S)$ is the analogous set in backwards time. The following
proposition follows immediately from the definition of $\Inv^\pm(N,S)$. 
\begin{prop}[Inclusion properties]
\label{prop:Inv-M-N}
  Assume that $M\subset N$ are two isolating sets for an isolated
  invariant set~$S$. Then $\Inv^\pm(M,S)\subset \Inv^\pm(N,S)$.
\qed
\end{prop}
In addition, the above two sets have interesting topological properties,
and they can be used to reconstruct an isolated invariant set~$S$, as
the next result shows.
\begin{prop}[Topological properties]
\label{prop:Inv-top-prop}
Assume that $S\subset N\subset X$, that~$S$ is an invariant set,
and that~$N$ is closed. Then the set~$\Inv^-(N,S)$ is closed
and~$\Inv^+(N,S)$ is open in~$N$. If in addition~$N$ isolates~$S$,
then one also has
\begin{equation} \label{eq:Inv-top-prop}
  \Inv^-(N,S) \cap \Inv^+(N,S) = S,
\end{equation}
and the isolated invariant set~$S$ is locally closed in~$X$,
that is $S$ is a difference of two closed sets in $X$ (see \cite[Problem 2.7.1]{En1989}).
\end{prop}
\begin{proof}
Denote $N^-:=\Inv^-(N,S)$ and $N^+:=\Inv^+(N,S)$. In order to prove
that~$N^-$ is closed take a $y\in\cl N^-$. Then $y\in\cl y'$ for some
$y'\in N^-$. Hence, we may take a path $\sigma\in\Path(N)$ from a point in $S$ to $y'$. Since~$S$ is invariant, without loss of generality we may assume that $|\sigma|\geq 2$. Since $F$ has closed values, replacing $y'$ by $y$ in $\sigma$ one obtains a new path, so $y\in N^-$. This proves that $N^-$ is indeed closed.


To see that the set~$N^+$ is open in $N$, choose any
$x\in\opn_N N^+ = N\cap\opn N^+$. Then $x\in\opn x'$ for some
$x'\in N^+$. Let $\sigma\in\Path(N)$ be a path from $x'$ to some point in $S$. Since $F$ is lower semicontinuous with closed values, $F(x')\subset F(x)$. Thus, replacing $x'$ by $x$ in $\sigma$ gives another path, so $x\in N^+$. Therefore $N\cap\opn N^+ \subset N^+$, so $N^+$ is open in $N$.


Finally, the inclusion $S\subset N^-\cap N^+$ is obvious.
Suppose now that~$N$ isolates~$S$. To see the opposite inclusion, let
$x\in N^-\cap N^+$ be arbitrary. Then there exist a
path in $N$ from a point in $S$ to $x$ and a path in~$N$ from~$x$ to a
point in~$S$. Concatenation of these gives a path in~$N$ through~$x$, and with endpoints
in~$S$. Hence, since~$N$ isolates~$S$, we obtain $x\in S$
and~(\ref{eq:Inv-top-prop}) holds. Moreover,
the representation~(\ref{eq:Inv-top-prop}) shows that~$S$ can be
written as
\begin{displaymath}
  S = N^- \setminus \left(N \setminus N^+ \right).
\end{displaymath}
Since $N^-$ and~$N \setminus N^+$ are closed, $S$ is locally closed in~$X$.
\qed
\end{proof}

\medskip
The above result shows that also in the multivalued map case, isolated
invariant sets necessarily have to be locally closed. This is reminiscent
of the situation in the multivector case~\cite{lipinski:etal:23a}, and it
provides a sufficient condition for recognizing invariant sets which are
not isolated invariant. In fact, this criterion does not make any reference
to an associated isolating set~$N$. For example, one can easily see that
the set $S_1 \cup S_3$ in Example~\ref{ex:simplecvf} is not locally closed,
and therefore it cannot be an isolated invariant set.

We now turn our attention to the existence of index pairs for isolated
invariant sets. For this, we need the following definition, as well as the
subsequent result.
\begin{defn}[Standard index pair]
\label{defpn}
%
%
Given an isolating set~$N$ for an isolated invariant set~$S$, we define
the {\em standard index pair\/} $P^N=(P^N_1,P^N_2)$ by
\begin{displaymath}
   P^N_1 := \Inv^-(N,S)
   \qquad\mbox{ and }\qquad
   P^N_2 := P^N_1\setminus\Inv^+(N,S).
\end{displaymath}
If we want to explicitly emphasize the dependence of the
index pair on the isolated invariant set~$S$, we also write
$P^{S,N}=(P^{S,N}_1,P^{S,N}_2)$ instead of $P^N=(P^N_1,P^N_2)$.
\end{defn}
\begin{thm}[Existence of saturated index pair]
\label{thm:ip-existence}
   Assume that $N\subset X$ is an isolating set for an isolated
   invariant set~$S$. Then $P^N$ is a saturated index pair for~$S$ in~$N$.
\end{thm}
\begin{proof}
It follows from Proposition~\ref{prop:Inv-top-prop} that the sets~$P^N_1$
and~$P^N_2$ are both closed. Moreover, property~(IP1) is a straightforward
consequence of the definition of the sets~$\Inv^-(N,S)$ and~$\Inv^+(N,S)$.
From~\eqref{eq:Inv-top-prop} we obtain $S = P^N_1 \setminus P^N_2$, which
establishes both~(IP3) and the fact that~$P^N$ is saturated, once
condition~(IP2) has been proved.

Thus, it remains to verify that property~(IP2) is satisfied. For this, assume
to the contrary that there exists an element $y\in (P_1^N \cap \cl (F(P_1^N)
\setminus N)) \setminus P^N_2$. This implies that $y\in P^N_1\setminus P^N_2=S$,
and there exists $y'\in F(P^N_1)\setminus N$ such that $y\in \cl y'$.
Let $x\in P^N_1$ be such that $y'\in F(x)$. Then $y\in\cl y'\subset\cl F(x)=F(x)$,
since $F$ has closed values. In view of $x\in P^N_1=\Inv^-(N,S)$, there exists a
path $\sigma\in\Path(N)$ such that $\lep{\sigma}\in S$ and $\rep{\sigma}=x$. It
follows that $\sigma\cdot y$ is a path in~$N$ with endpoints in~$S$, and
therefore~(IS1) yields~$x\in S$. This in turn implies $y'\in F(x)\subset F(S)$.
Thus, one obtains $y\in S\cap\cl y'\subset S\cap \cl(F(S)\setminus N)$,
which contradicts~(IS2).
\qed
\end{proof}

\medskip
The standard index pair~$P^N$ that can be associated with every
isolated invariant~$S$ with isolating set~$N$ will be important for our further
considerations. Yet, as we pointed out earlier, this is only one possible 
choice among many. 
In particular, although the standard index pair is sufficient to define the Conley index, 
the flexibility in choosing index pairs matters when addressing properties of the index, 
for instance continuation (see Sec.~\ref{sec:continuation}).
While the collection of index pairs will be further studied
in the next section, we close this one with a simple observation.
\begin{prop}[Inclusion property of standard index pairs]
\label{prop:ip-M-N}
  Assume $M\subset N$ are two isolating sets for an isolated invariant set~$S$.
  Then the associated standard index pairs satisfy $P^M_i\subset P^N_i$ for $i=1,2$.
\end{prop}
\begin{proof}
The inclusion $P^M_1\subset P^N_1$ follows immediately from
Proposition~\ref{prop:Inv-M-N}. On the other hand, in view
of~\eqref{eq:Inv-top-prop} we have $P^M_2 = P^M_1 \setminus S
\subseteq P^N_1 \setminus S = P^N_2$.
\qed
\end{proof}

\medskip
To close this section, we briefly return to our previous two examples
and present the standard index pairs for selected isolated invariant
sets.
\begin{ex}[Sample standard index pairs]
\label{ex:standardindexpairs}
{\em
For the two simple multivalued maps~$F : X \mto X$ and~$G : X \mto X$ from
Examples~\ref{ex:simplecvf} and~\ref{ex:reflectedcvf}, respectively, on
the finite topological space $X = \{ A, B, C, AB, AC, BC, ABC \}$, one
can easily determine the associated standard index pairs. Recall that in
Example~\ref{ex:simplecvf} we used the closed sets $N_1 = \{ A, B, C \}$,
$N_2 = \{ A, B, C, AB, BC, AC \}$, and~$N_3 = X$ as respective isolating
sets for the three isolated invariant sets~$S_1$, $S_2$, and~$S_3$ given
below. This leads to the standard index pairs
\begin{equation} \label{ex:standardindexpairs1}
  \begin{array}{llclcl}
    \DS P^{S_1,N_1}_1 = N_1, &
      \DS P^{S_1,N_1}_2 = \emptyset &
      \;\mbox{ for }\; &
      \DS S_1 = \{ A, B, C \} &
      \;\mbox{ in }\; & N_1, \\[0.5ex]
    \DS P^{S_2,N_2}_1 = N_2, &
      \DS P^{S_2,N_2}_2 = N_1 &
      \;\mbox{ for }\; &
      \DS S_2 = \{ AB, BC, AC \} &
      \;\mbox{ in }\; & N_2, \\[0.5ex]
    \DS P^{S_3,N_3}_1 = N_3, &
      \DS P^{S_3,N_3}_2 = N_2 &
      \;\mbox{ for }\; &
      \DS S_3 = \{ ABC \} &
      \;\mbox{ in }\; & N_3.
  \end{array}
\end{equation}
For example, in order to establish the second standard index pair in
this list, note that $\Inv^-(N_2,S_2) = \{ A, B, C, AB, BC, AC \}$ and
$\Inv^+(N_2,S_2) = \{ AB, BC, AC \}$, which immediately yields the
above form for~$P^{S_2,N_2}$.

We now turn our attention to Example~\ref{ex:reflectedcvf}. In this case,
we defined the closed sets $M_1 = \{ A \}$, $M_2 = \{ B, C \}$, $M_3 =
\{ B, C, BC \}$, $M_4 = \{ A, B, C, AB, AC \}$, as well as $M_5 = X$, as
respective isolating sets for the isolated invariant sets~$R_k$ given 
below. More precisely, one obtains the standard index pairs
\begin{equation} \label{ex:standardindexpairs2}
  \begin{array}{llclcl}
    \DS P^{R_1,M_1}_1 = M_1, &
      \DS P^{R_1,M_1}_2 = \emptyset &
      \;\mbox{ for }\; &
      \DS R_1 = \{ A \} &
      \;\mbox{ in }\; & M_1, \\[0.5ex]
    \DS P^{R_2,M_2}_1 = M_2, &
      \DS P^{R_2,M_2}_2 = \emptyset &
      \;\mbox{ for }\; &
      \DS R_2 = \{ B, C \} &
      \;\mbox{ in }\; & M_2, \\[0.5ex]
    \DS P^{R_3,M_3}_1 = M_3, &
      \DS P^{R_3,M_3}_2 = M_2 &
      \;\mbox{ for }\; &
      \DS R_3 = \{ BC \} &
      \;\mbox{ in }\; & M_3, \\[0.5ex]
    \DS P^{R_4,M_4}_1 = M_4, &
      \DS P^{R_4,M_4}_2 = M_1 \cup M_2 &
      \;\mbox{ for }\; &
      \DS R_4 = \{ AB, AC \} &
      \;\mbox{ in }\; & M_4, \\[0.5ex]
    \DS P^{R_5,M_5}_1 = M_5, &
      \DS P^{R_5,M_5}_2 = M_3 \cup M_4 &
      \;\mbox{ for }\; &
      \DS R_5 = \{ ABC \} &
      \;\mbox{ in }\; & M_5.
  \end{array}
\end{equation}
Thus, we have identified the standard index pairs for all
isolated invariant sets contained in the Morse decompositions
shown in Figures~\ref{fig:simplemorse} and~\ref{fig:reflectedmorse}.
}
\end{ex}
%

%
\subsection{Properties of index pairs}
\label{sec:indexpairprop}
In the last section, we introduced the notion of an index pair
$P = (P_1,P_2)$ associated with an isolated invariant set~$S$ and its
isolating set~$N$. These index pairs will prove to be central for the
definition of the Conley index. Yet, as we already mentioned several times,
index pairs are not unique, and the present section collects results on
the construction of a variety of index pairs. These results will be crucial
for the next section, which introduces the Conley index.

In the following, we assume that~$N$ is an isolating set for the isolated
invariant set~$S$. If $P=(P_1,P_2)$ and $Q=(Q_1,Q_2)$ are two topological
pairs, we use the abbreviation $P\subset Q$ for the validity of the two
inclusions $P_1\subset Q_1$ and $P_2\subset Q_2$. Furthermore, by~$P\cap Q$
we denote the pair $(P_1\cap Q_1, P_2\cap Q_2)$. We begin by showing that
index pairs are closed under intersection.
\begin{lem}[Intersection preserves index pairs]
\label{inter1}
If~$P$ and~$Q$ are two index pairs for an isolated invariant set~$S$ in an
isolating set~$N$, then so is~$P\cap Q$.
\end{lem}
\begin{proof}
Applying property~(IP1) of~$P$ we get $F(P_i\cap Q_i)\cap N\subset
F(P_i)\cap N\subset P_i$. Similarly, we obtain $F(P_i\cap Q_i)\cap
N\subset Q_i$. Therefore, $F(P_i\cap Q_i)\cap N\subset P_i\cap Q_i$
for $i = 1,2$, which proves the inclusions in~(IP1) for~$P\cap Q$.

As for the second property~(IP2) of an index pairs, we observe that
since both~$P$ and~$Q$ satisfy it, one obtains the inclusions
\begin{eqnarray*}
  P_1 \cap Q_1 \cap \cl (F(P_1\cap Q_1)\setminus N) &\subset &
    P_1 \cap \cl (F(P_1)\setminus N) \cap Q_1 \cap \cl (F(Q_1)\setminus N) \\[0.5ex]
  & \subset & P_2\cap Q_2.
\end{eqnarray*}
It remains to establish~(IP3). First observe that in view of~(IP3)
for both~$P$ and~$Q$ we have
\begin{displaymath}
  S \;\subset\;
  (P_1\setminus P_2)\cap (Q_1\setminus Q_2) \;\subset\;
  (P_1\cap Q_1)\setminus (P_2\cap Q_2).
\end{displaymath}
Therefore, $S=\Inv S\subset \Inv((P_1\cap Q_1)\setminus (P_2\cap Q_2))$.
To prove the opposite inclusion, assume to the contrary that there exists
$y\in \Inv((P_1\cap Q_1)\setminus (P_2\cap Q_2))\setminus S$. Moreover,
let $\sigma:\ZZ\to (P_1\cap Q_1)\setminus (P_2\cap Q_2)$ be a full
solution through $y$. Then there has to exist an index $p\in\ZZ$ such
that $\sigma(p)\in P_2$, because otherwise we obtain the inclusion
$\im\sigma\subset \Inv(P_1\setminus P_2)=S$ in view of~(IP3) for~$P$.
In addition, due to~(IP1) for~$P$, together with $\im\sigma \subset
P_1 \subset N$, one has to have $\sigma(r)\in P_2$ for every $r\geq p$.
Symmetrically, there exists an index $q\in \ZZ$ such that $\sigma (r)
\in Q_2$ for all $r\geq q$. In particular, this implies
$\sigma (\max \{p,q\}) \in P_2\cap Q_2$, a contradiction.
\qed
\end{proof}

\medskip
The next two results introduce a few ways for constructing new index
pairs from two given nested ones.
\begin{lem}[New index pairs from nested ones]
\label{inter2}
If $P\subset Q$ are index pairs in~$N$ for an isolated invariant set~$S$,
then so are~$(P_1,P_1\cap Q_2)$ and~$(P_1\cup Q_2,Q_2)$.
\end{lem}
\begin{proof}
Let us start with the first pair $(P_1,P_1\cap Q_2)$. The verification
of property~(IP1) is straightforward. Observe that in view of~(IP2)
for the index pair~$P$ we get $P_1\cap \cl (F(P_1)\setminus N) \subset
P_2\subset P_1\cap Q_2$, and therefore~(IP2) holds. To establish~(IP3),
we observe that due to~(IP3) for both~$P$ and~$Q$ one has
\begin{eqnarray*}
  S & = & \Inv S
    \; \subset \; \Inv((P_1\setminus P_2)\cap(Q_1\setminus Q_2))
    \; \subset \; \Inv(P_1\setminus Q_2) \\[0.5ex]
  & = & \Inv(P_1\setminus(P_1\cap Q_2))
    \; \subset \; \Inv(P_1\setminus P_2) \; = \; S.
\end{eqnarray*}
Hence, $\Inv(P_1\setminus (P_1\cap Q_2))=S$, which completes
the proof that $(P_1,P_1\cap Q_2)$ is indeed an index pair.

Consider now the second pair~$(P_1\cup Q_2,Q_2)$. As before, the verification
of~(IP1) is straightforward. In order to establish~(IP3) we observe that as seen
above
\begin{displaymath}
  S=\Inv(P_1\setminus Q_2)= \Inv((P_1\cup Q_2)\setminus Q_2).
\end{displaymath}
Finally, in order to verify~(IP2) for $(P_1\cup Q_2,Q_2)$ we note that
\begin{displaymath}
  (P_1 \cup Q_2) \cap \cl (F(P_1\cup Q_2) \setminus N) \; \subset \;
  Q_1 \cap \cl (F(Q_1) \setminus N) \; \subset \; Q_2,
\end{displaymath}
which yields~(IP2) for the second pair and completes the proof.
\qed
\end{proof}

\medskip
The second lemma is concerned with a useful construction of new index
pairs which includes the action of~$F$ itself. For this, suppose we are 
given two index pairs~$P$ and~$Q$ for an isolated invariant set~$S$ in~$N$,
and such that $P\subset Q$. We then define a topological pair of sets
$G(P,Q)=(G_1(P,Q),G_2(P,Q))$ by
\begin{displaymath}
  G_i(P,Q) = P_i\cup(F(Q_i)\cap N)
  \quad\mbox{ for }\quad
  i=1,2.
\end{displaymath}
Note that we always have $G_2(P,Q)\subset G_1(P,Q) \subset N$, as required
by a topological pair, and that~$G_i(P,Q)$ is closed for $i=1,2$. The latter
fact is due to the closedness of the values of~$F$. While in general the
pair~$G(P,Q)$ is not an index pair for~$S$ in~$N$, the following result
gives sufficient conditions, as well as a number of other useful properties.
\begin{lem}[Properties of the pair~$G(P,Q)$]
\label{lem:G_PQ}
Let $P\subset Q$ be two index pairs for the isolated invariant set~$S$
in~$N$, and let~$G=G(P,Q)$ be defined as above. Then we have the following
properties.
\begin{itemize} \setlength{\itemsep}{3pt}
\item[(i)] $P\subset G\subset Q$.
\item[(ii)] $P_i=Q_i$ implies $P_i=G_i=Q_i$, for $i=1,2$.
\item[(iii)] If $P_1=Q_1$ or $P_2=Q_2$ then $G$ is an index pair in $N$.
\item[(iv)] $F(Q_i)\cap N\subset G_i$, for $i=1,2$.
\item[(v)]  If $P_{3-i}=Q_{3-i}$ and $G_i=Q_i$, then $P_i=Q_i$ for $i=1,2$.
\end{itemize}
\end{lem}
\begin{proof}
The first inclusion in (i) is obvious. The second one follows from~(IP1)
for~$Q$. Moreover, property~(ii) is an immediate consequence of~(i).

In order to prove property (iii), let us begin with property~(IP1).
Its verification does not require the hypothesis of~(iii), since we
have
\begin{displaymath}
  F(G_i) \cap N \; = \;
  (F(P_i) \cap N) \cup (F(F(Q_i)\cap N) \cap N) \; \subset \;
  P_i \cup (F(Q_i) \cap N) \; = \; G_i
\end{displaymath}
in view of~(IP1) applied to~$P$ and~$Q$.

If $P_1=Q_1$, then we have $G_1\cap \cl (F(G_1)\setminus N)=P_1\cap
\cl (F(P_1)\setminus N) \subset P_2 \subset G_2$ by~(i), (ii), and~(IP2)
for~$P$, and this establishes~(IP2) for~$G$ in this case. On the other 
hand, if the equality $P_2=Q_2$ holds, then
\begin{displaymath}
  G_1 \cap \cl (F(G_1)\setminus N) \; \subset \;
  Q_1 \cap \cl (F(Q_1)\setminus N) \; \subset \;
  Q_2=G_2,
\end{displaymath}
in view of~(IP2) for~$Q$, (i), and~(ii). This proves~(IP2) for~$G$
also in this case.

In order to verify property~(IP3), observe that by~(IP3) applied to~$P$
and~$Q$ one obtains $S\subset (P_1\setminus P_2) \cap (Q_1\setminus Q_2) =
P_1\setminus Q_2\subset G_1\setminus G_2$, and this in turn immediately
yields $S=\Inv S\subset\Inv ( G_1\setminus G_2)$. According to property~(i)
we have the inclusion $G_1\setminus G_2\subset Q_1\setminus P_2$. Hence,
if $P_1=Q_1$, we obtain $G_1\setminus G_2\subset P_1\setminus P_2$,
and~(IP3) applied to~$P$ further implies $\Inv(G_1\setminus G_2)\subset
\Inv(P_1\setminus P_2)=S$. Similarly, if instead the equality $P_2=Q_2$
holds, then one obtains $G_1\setminus G_2\subset Q_1\setminus Q_2$,
and~(IP3) applied to~$Q$ furnishes $\Inv(G_1\setminus G_2)\subset
\Inv(Q_1\setminus Q_2)=S$. Altogether, we get the inclusion
$\Inv(G_1\setminus G_2) \subset S$, which completes the proof
of~(IP3) for~$G$, and establishes the latter as an index pair,
as claimed in~(iii).

Since property~(iv) is obvious, it remains to establish~(v). For this,
fix $i\in\{1,2\}$ and assume that the identities $P_{3-i}=Q_{3-i}$
and~$G_i=Q_i$ are satisfied. We want to show that $P_i=Q_i$. Since
$P_i\subset Q_i$ by assumption, we only need to verify the
inclusion $Q_i\subset P_i$.

Thus, take an arbitrary point $y\in Q_i$. We will begin by constructing
recursively a function $\sigma:\ZZ_-\to Q_i$ as follows, where~$\ZZ_-$
denotes the set of all nonpositive integers. We set
$\sigma(0):=y\in Q_i$. Assuming $\sigma(-k)\in Q_i$ has already been
defined for $k\in\NN_0$, we consider two cases to define~$\sigma(-k-1)$.
If we have $\sigma(-k)\in P_i$, then we define $\sigma(-k-1):=\sigma(-k)$.
If instead we have $\sigma(-k)\not\in P_i$, then one obtains
from the assumption $Q_i = G_i$ and the above definition $G_i =
P_i\cup(F(Q_i)\cap N)$ that $\sigma(-k)\in F(Q_i)$, and we can
select an element~$\sigma(-k-1)\in Q_i$ which satisfies the inclusion
$\sigma(-k)\in F(\sigma(-k-1))$.

We claim that $\im\sigma\cap P_i\neq\emptyset$. Assume the contrary.
Then $\sigma:\ZZ_-\to Q_i\setminus P_i$ is a solution. Since the
space~$X$ is finite, we can therefore find indices $m,n\in\ZZ_-$
such that $m<n$ and $\sigma(m)=\sigma(n)$. Thus, the point~$\sigma(m)$
lies on a periodic solution in the set difference $Q_i\setminus P_i$.
But then we have $\sigma(m)\in\Inv (Q_i\setminus P_i)$. Consider now 
first the case $i=1$. Then we have $P_2=Q_2$ and $\sigma:\ZZ_-
\to Q_1\setminus P_1$, as well as the inclusion $Q_1\setminus P_1\subset
Q_1\setminus P_2=Q_1\setminus Q_2$. Hence, using property~(IP3) applied
to~$Q$ one obtains that $\sigma(m) \in \Inv (Q_1\setminus Q_2) = S
\subset P_1$, which contradicts our assumption that $\im\sigma\cap P_1
=\emptyset$. Consider now the second case $i=2$. Then one has $P_1=Q_1$
and $\sigma:\ZZ_-\to Q_2\setminus P_2$, as well as $Q_2\setminus
P_2\subset Q_1\setminus P_2=P_1\setminus P_2$. Hence, we get from~(IP3)
applied to~$P$ that $\sigma(m)\in\Inv (P_1\setminus P_2)=S\subset
Q_1\setminus Q_2$. Therefore, $\sigma(m)\not\in Q_2$, again a
contradiction. Thus, we established $\im\sigma\cap P_i\neq\emptyset$.

With this we can immediately complete the proof of~(v). According to the last
paragraph, the index $m:=\max\setof{k\in\ZZ_-\mid \sigma(k) \in P_i}$
is well-defined. We cannot have $m<0$, because in that case one obtains
$\sigma(m+1) \in F(\sigma (m)) \subset F(P_i)$, and due to~(IP1)
applied to $P$ one further gets $\sigma(m+1) \in Q_i \cap F(P_i) \subset
N \cap F(P_i) \subset P_i$, which is a contradiction. Hence, we have to
have $m=0$, and thus $y=\sigma(0)\in P_i$. This completes the proof of
the lemma.
\qed
\end{proof}

\medskip
The next result shows that for nested index pairs $P \subseteq Q$ which
satisfy $P_1 = Q_1$ or $P_2 = Q_2$, it is always possible
to construct a sequence of index pairs between them with certain
mapping properties. While the specifics of this lemma might seem 
strange at first sight, it is essential for proving that the Conley
index computation is independent of the underlying index pair.
\begin{lem}[Interpolating between nested index pairs]
\label{lem:seq-Qi}
Let $P\subset Q$ be index pairs for an isolated invariant set~$S$
in~$N$ such that either $P_1=Q_1$ or $P_2=Q_2$. Then there exists
a sequence of index pairs for $S$ in $N$
\begin{displaymath}
  P=Q^n\subset Q^{n-1}\subset\cdots\subset Q^1\subset Q^0=Q
\end{displaymath}
which satisfy the following:
\begin{itemize} \setlength{\itemsep}{3pt}
\item[(a)] $P_i=Q_i$ implies $Q^k _i=P_i=Q_i$ for all $k=1,2,\dots, n-1$
           and $i=1,2$,
\item[(b)] $F(Q^k _i)\cap N\subset Q^{k+1} _i$ for all $k=0,1,\dots, n-1$
           and $i=1,2$.
\end{itemize}
\end{lem}
\begin{proof}
Define the index pairs~$Q^k$ recursively by $Q^0:=Q$ and $Q^{k+1}:=G(P,Q^k)$
for $k\in\NN$. Using Lemma~\ref{lem:G_PQ}(i), (ii) and (iii), together with
induction on~$k$, one can easily show that the family~$\{Q^k\}$ forms a
decreasing sequence of index pairs with respect to~$k$ which satisfies
property~(a). In addition, Lemma~\ref{lem:G_PQ}(iv) implies that they also
satisfy property~(b) for all $k\in\NN_0$. Finally, since~$X$ is finite,
there has to be an $n\in \NN_0$ such that $Q^n=Q^{n+1}=G(P,Q^n)$, and an
application of Lemma~\ref{lem:G_PQ}(v) shows that then $Q^n=P$.
\qed
\end{proof}

\medskip
For the remainder of this section, we briefly introduce and study a
topological pair which can be associated with an index pair, and which
plays a crucial role for the definition of the {\em index map\/} in the
next section.

To define this topological pair, we again let $P=(P_1,P_2)$ denote an
index pair for an isolated invariant set~$S$ in the isolating set~$N$.
Then we define $\bar{P}:=(\bar{P}_1,\bar{P}_2)$ via
\begin{displaymath}
  \bar{P}_i := P_i \cup \cl (F(P_1)\setminus N)
  \quad\mbox{ for }\quad i = 1,2.
\end{displaymath}
Notice that the new topological pair~$\bar{P}$ extends the index pair~$P$
by adding the closure of the images of~$P_1$ under~$F$ that lie outside
of~$N$. The resulting pair~$\bar{P}$ still consists of closed sets, but
in general it is no longer an index pair. Nevertheless, it will allow us
to study the action of~$F$ on~$P$ on the homological level in the next
section. For now, we note the following proposition.
\begin{prop}[The extended topological pair~$\bar{P}$]
\label{prop:index-map}
Assume that $P=(P_1,P_2)$ is an index pair for the isolated invariant
set~$S$ in an isolating set~$N$. Then the following hold for the extended
topological pair~$\bar{P}$ defined above:
\begin{eqnarray}
  && P\subset \bar{P},\label{eq:index-map-1}\\
  && F(P)=(F(P_1),F(P_2))\subset \bar{P},\label{eq:index-map-2}\\
  && \bar{P}_1\setminus\bar{P}_2 = P_1\setminus P_2.\label{eq:index-map-3}
\end{eqnarray}
\end{prop}
\begin{proof}
As we already mentioned, property~\eqref{eq:index-map-1} follows directly
from the definition of~$\bar{P}$. To see~\eqref{eq:index-map-2}, note that
in view of~(IP1) we have $F(P_i)\setminus P_i \subset F(P_i) \setminus N$,
and therefore $F(P_i)\subset P_i\cup (F(P_i)\setminus P_i) \subset
P_i\cup (F(P_i)\setminus N) \subset \bar{P}_i$. Finally, observe that
property~(IP2) implies
\begin{eqnarray*}
  \bar{P_1} \setminus \bar{P_2} & = &
    (P_1\cup \cl (F(P_1)\setminus N)) \setminus
    (P_2\cup \cl (F(P_1)\setminus N)) \\[0.5ex]
  & = & P_1\setminus P_2\setminus \cl (F(P_1)\setminus N)
    = P_1\setminus P_2,
\end{eqnarray*}
and this completes the proof of the proposition.
\qed
\end{proof}

\medskip
As our final result of this section, we consider the extended topological
pairs of the standard index pairs introduced in Definition~\ref{defpn}.
More precisely, we consider the situation of nested isolating sets for
the same isolated invariant set~$S$.
\begin{prop}[The extended topological pair for standard index pairs]
\label{prop:barip-M-N}
Assume that the closed sets $M\subset N$ are two isolating sets for the
same isolated invariant set~$S$. Then the inclusion $\overline{P^M} \subset
\overline{P^N}$ holds.
\end{prop}
\begin{proof} 
We first establish the validity of $\overline{P^M_1}\subset \overline{P^N_1}$.
Using Proposition~\ref{prop:ip-M-N} one obtains the inclusion
\begin{eqnarray*}
  \overline{P^M_1} & = &
    P^M_1\cup \cl (F(P^M_1)\setminus M) \; \subset \;
    P^N_1\cup \cl (F(P^N_1)\setminus M) \\[0.5ex]
  & = & P^N_1 \cup \cl (F(P^N_1)\setminus N) \cup \cl ((F(P^N_1)\cap N)
    \setminus M) \\[0.5ex]
  & \subset &
    P^N_1\cup \cl (F(P^N_1)\setminus N) \; = \;
    \overline{P^N_1},
\end{eqnarray*}
where the last inclusion follows from~(IP1) and the fact that~$P^N_1$
is closed.

It remains to show that $\overline{P^M_2} \subset \overline{P^N_2}$.
For this, let $y\in \overline{P^M_2}=P^M_2 \cup \cl(F(P^M_1)\setminus M)$.
If in fact we have $y\in P^M_2$, then an application of
Proposition~\ref{prop:ip-M-N} immediately implies that
$y\in P^N_2\subset \overline{P^N_2}$. Suppose therefore that we have
$y\in \cl (F(P^M_1)\setminus M)$ and $y\notin P^N_2$. Furthermore, let
$y'\in F(P^M_ 1) \setminus M$ be such that $y\in \cl y'$.

We now claim that $y' \not\in N$. To verify this, assume to the contrary
that $y'\in N$. Let $x\in P^M_1$ be such that $y'\in F(x)$. If $x\in P^M_2$,
then $y'\in F(P^M_2)\subset F(P^N_2)$, and therefore $y'\in F(P^N_2)\cap N
\subset P^N_2$ by~(IP1), as well as $y\in \cl y' \subset  \cl P^N_2=P^N_2$,
a contradiction. Thus we have to have $x\notin P^M_2$, and this yields
$x\in P^M_1\setminus P^M_2=S$ by Theorem \ref{thm:ip-existence}. Hence,
$y'\in F(S)\setminus M$ and $y\in \cl(F(S)\setminus M)$.

Since we assumed $y'\in N$, one obtains $y'\in F(P^M_1)\cap N \subset
F(P^N_1)\cap N \subset P^N_1$, and together with the closedness of~$P^N_1$
this further implies $y\in P^N_1$. This in turn shows that the inclusion
$y\in P^N_1\setminus P^N_2=S$ holds, by Theorem~\ref{thm:ip-existence}.
However, this finally furnishes $y\in S\cap \cl (F(S)\setminus M)$, which
contradicts~(IS2). Thus, we deduce that our assumption on~$y'$ was wrong
and we actually have $y'\notin N$.

With this in hand the proof of the second inclusion can easily be finished.
We now have $y'\in F(P^M_1)\setminus N \subset F(P^N_1)\setminus N$, as well
as $y\in \cl(F(P^N_1)\setminus N) \subset \overline{P^N_2}$. 
\qed
\end{proof}

\medskip
We close this section by deriving the extended topological pairs for
the index pairs in Example~\ref{ex:standardindexpairs}.
\begin{ex}[Sample extended topological pairs]
\label{ex:extendedtoppairs}
{\em
We leave it to the reader to verify that all eight standard
index pairs given in~(\ref{ex:standardindexpairs1})
and~(\ref{ex:standardindexpairs2}) are in fact equal to their
extended topological pairs as defined above. In every one of
these cases, if~$S$ is an isolated invariant set in an 
isolating set~$N$, we have both $P^{S,N}_1 = N$, as well as
either $F(N) \subset N$ or $G(N) \subset N$, respectively,
depending on which multivalued map is considered. From this
our claim follows immediately.
}
\end{ex}
%
%
%
\section{Definition of the Conley index}
\label{sec:conleyindex}
With this section, we finally turn our attention to the Conley index
for isolated invariant sets. For this, we first introduce the index
map based on an index pair in Section~\ref{subsec:indexmap}, which
transfers the action of the multivalued map~$F$ restricted to the
index pair to the algebraic level in terms of homology. Clearly, this
map will depend on the chosen index pair, and the remainder of the section
is aimed at deriving an index definition from the index map which only
depends on the isolated invariant set. Our approach relies on the notion 
of normal functor, which is introduced in Section~\ref{subsec:normalfunctor}.
Finally, Section~\ref{subsec:conleyindex} combines both notions to define
the Conley index and prove that it is well-defined. In addition, we compute
the Conley indices for the isolated invariant sets in our earlier examples.
In contrast to the previous section, we need to impose an additional
condition on the underlying multivalued map. In this section the multivalued
map $F:X\mto X$ will be assumed to be lower semicontinuous with closed and {\em acyclic\/}
values. The additional acyclicity assumption is needed in order to
obtain induced maps in homology. All the homology groups are considered
with coefficients in a fixed ring~$R$. 
%
%
\subsection{The index map}
\label{subsec:indexmap}
The basic idea of the Conley index in this paper is to lift information
from the multivalued map $F:X\mto X$ close to an isolated invariant
set~$S$ to the setting of homology. On the level of the phase space, this
is accomplished by considering the relative homology~$H_*(P) = H_*(P_1,P_2)$
of an index pair for~$S$, and on the level of the map~$F$ by the associated
{\em index map\/}~$I_P$, which is an endomorphism of~$H_*(P)$. The latter
map should of course in some way lift the dynamics of~$F$ to the
homology level.

Passing from a multivalued map to a map induced in homology is slightly
more involved than the classical map setting, and we begin by reviewing
the necessary approach. As it was already said in Section \ref{sec:prelim}, 
every lower semicontinuous map with closed values is in fact {\em strongly lower semicontinuous (slsc)\/}. Recall that a multivalued
map $F:X\mto Y$ between finite $T_0$ spaces is called strongly lower
semicontinuous, if $x\in \cl x'$ implies $F(x) \subset F(x')$. If
in addition~$F:X\mto Y$ has acyclic values, then it induces a
homomorphism $F_*:H_*(X)\to H_*(Y)$ in homology. More precisely,
in~\cite[Proposition~4.7]{barmak:etal:20a} it is shown that if~$F$ is
identified with its graph $F\subseteq X\times Y$, then the restriction
$p_1:F\to X$ of the projection onto the first coordinate is an isomorphism
in homology in every degree, and therefore one can define the induced map
in homology as $F_* = (p_2)_*(p_1)_*^{-1}$ with $p_2:F\to Y$ denoting the restriction of the other projection.  
Although the results in \cite{barmak:etal:20a} are stated only for integer coefficients, it is easy to see that the same results hold for homology with coefficients in an arbitrary ring (see \cite[Theorem~2.2]{barmak:etal:20a}).

As we mentioned earlier, the index map will be a homological version
of the action of~$F$ on a given index pair, and it is therefore not
surprising that we have to recall a few notions about maps of pairs.
A multivalued map $F:(X,A)\multimap (Y,B)$ between pairs of finite~$T_0$
spaces is a multivalued map $F:X\multimap Y$ which satisfies the inclusion
$F(a)\subseteq B$ for every $a\in A$. We say that $F:(X,A)\multimap (Y,B)$
is slsc (or with closed values, or with acyclic values) if $F:X\multimap Y$
has the respective property. Suppose that $F:(X,A) \multimap (Y,B)$ is slsc
with acyclic values. Then the restriction $F|^B_A:A\multimap B$ is also slsc
with acyclic values, and its graph is a subspace of~$F$. Since the
projections~$F\to X$ and~$F|^B_A\to A$ induce isomorphisms in homology,
by the long exact sequence of homology and the five lemma, so does the
projection of pairs $p_1:(F,F|^B_A)\to (X,A)$. Finally, in view of these
preparations we can define the homomorphims $F_*:H_*(X,A)\to H_*(Y,B)$
by letting $F_*=(p_2)_*(p_1)^{-1}_*$ as before.

Before moving on to the definition of the index map, we need the following
two auxiliary results concerning maps in homology induced by compositions.
\begin{lem}[Homology map of compositions]
\label{lemma1}
Let $F:X\multimap Y$ and $G:Z\multimap Y$ be slsc multivalued maps with
acyclic values, and suppose that $f:X\to Z$ is a continuous map such
that $Gf=F$. Then we have $G_*f_*=F_*:H_*(X)\to H_*(Y)$. The same result
holds more generally, for pairs.
\end{lem}
\begin{proof}
Consider the following two commutative diagrams:
\[
\begin{diagram}
 \node{X}
 \arrow{s,l}{f}
 \node{Y}
\arrow{w,t,dd}{F}
\arrow{sw,r,dd}{G}
 \node{X}
 \arrow{s,l}{f}
 \node{F}
 \arrow{w,l}{p_1}
 \arrow{e,l}{p_2}
 \arrow{s,l}{f\times 1_Y}
 \node{Y}
 \\
 \node{Z}
 \node{}
 \node{Z}
 \node{G}
 \arrow{ne,b}{p_2}
 \arrow{w,b}{p_1}
\end{diagram}
\]
The commutativity of the first diagram implies that $f\times 1_Y:F\to G$
is well-defined, and this immediately leads to the second commutative
diagram. The result then follows by definition. For pairs we have the
exact same proof.
\qed
\end{proof}
\begin{lem}[Homology map of compositions]
\label{lemma2}
Let $F:Z\multimap X$ and $G:Z\multimap Y$ be slsc multivalued maps
with acyclic values, and let $f:Y\to X$ be a continuous map such
that $fG=F$. Then $f_*G_*=F_*:H_*(Z)\to H_*(X)$. The same result
holds more generally, for pairs.
\end{lem}
\begin{proof}
Similar to the last proof, consider the following commutative diagrams:
\[
\begin{diagram}
 \node{}
 \node{Y}
\arrow{sw,l,dd}{G}
 \arrow{s,r}{f}
 \node{}
 \node{G}
 \arrow{sw,l}{p_1}
 \arrow{s,r}{1_Z\times f}
 \arrow{e,l}{p_2}
 \node{Y}
 \arrow{s,r}{f}
 \\
 \node{Z}
 \node{X}
\arrow{w,b,dd}{F}
 \node{Z}
 \node{F}
 \arrow{w,b}{p_1}
 \arrow{e,b}{p_2}
 \node{X}
\end{diagram}
\]
The commutativity of the first diagram implies
that $1_Z\times f:Z\times Y \to Z\times X$ is well-defined.
This leads to the second commutative diagram, and the result
then follows by definition. For pairs we have the exact same proof.
\qed
\end{proof}

\medskip
As our last preparation we turn our attention briefly to the {\em strong
excision\/} property. For this, let~$(Y_1,Y_2)$ and~$(Z_1,Z_2)$ denote two
topological pairs of closed subspaces of a finite~$T_0$ space~$X$ such that
the inclusions $Y_i\subset Z_i$ hold for $i=1,2$, and that $Y_1\setminus
Y_2 = Z_1\setminus Z_2$. Then the inclusion $i:(Y_1,Y_2) \to (Z_1,Z_2)$
induces a homomorphism~$i_*$ between the relative homology
groups~$H_*(Y_1,Y_2)$ and~$H_*(Z_1,Z_2)$. In fact, the strong excision
property states that $i_*:H_*(Y_1,Y_2) \to H_*(Z_1,Z_2)$ is an isomorphism.
This result follows directly from the pair of McCord isomorphisms
$H_*(\kp(Y_1),\kp(Y_2))\to H_*(Y_1,Y_2)$ and $H_*(\kp(Z_1),\kp(Z_2))\to
H_*(Z_1,Z_2)$, where~$\kp$ is the functor which associates to each poset
its order complex (\cite[Corollary~1]{mccord:66a}). The hypotheses imply
that the chains in~$Y_1$ which are not in~$Y_2$ are the same as the chains
in~$Z_1$ not in~$Z_2$, and thus $i_* :C_*(\kp (Y_1),\kp(Y_2)) \to
C_*(\kp(Z_1),\kp(Z_2))$ is already an isomorphism of chain complexes,
and in particular an isomorphism in homology. 

After these preparations we can finally introduce the index map. In the rest
of the paper~$X$ will be a finite~$T_0$ topological space and $F:X\mto X$
will be lower semicontinuous with closed and acyclic values. The index map
lifts the action of the multivalued 
map~$F$ on an index pair~$P$ to the homological level. This
has to be done with care, since we usually do not have $F(P) \subseteq P$.
In fact, we will make use of the extended pair~$\bar{P}$ whose properties
where established in Proposition~\ref{prop:index-map}. More precisely,
let~$P$ be an index pair for an isolated invariant set~$S$ in the
isolating set~$N$. By applying Proposition~\ref{prop:index-map}, we
then immediately obtain both an inclusion induced isomorphism
$(\iota_P)_*: H_*(P)\to H_*(\bar{P})$ and a homomorphism
$(F_P)_*: H_*(P)\to H_*(\bar{P})$, where the latter is induced
by the multivalued map~$F_P=F|_{P_1}^{\overline{P}_1}: P\mto \overline{P}$.
This leads to the following definition.
\begin{defn}[Index map]
\label{def:indexmap}
%
%
Let~$P$ be an index pair for an isolated invariant set~$S$ in the isolating
set~$N$. Then the associated {\em index map\/} is the endomorphism
\begin{displaymath}
  I_P : H_*(P_1,P_2) \to H_*(P_1,P_2)
  \quad\mbox{ given by }\quad
  I_P:=(\iota_P)_*^{-1}\circ (F_P)_*,
\end{displaymath}
where we use the maps induced in homology by the restriction
$F_P=F|_{P_1}^{\overline{P}_1}: P\mto \overline{P}$
and the inclusion $\iota_P : P \to \bar{P}$.
%
%
\end{defn}
\begin{rem} \label{index-mapnew}
Note that the definition of~$\bar{P}$ makes sense and the conclusion of
Proposition~\ref{prop:index-map} remains true even if $P=(P_1,P_2)$ is
merely a pair of closed subspaces of~$X$, and if~$N$ is a closed subspace
of~$X$ such that $P_2\subset P_1 \subset N$ and conditions~(IP1) and~(IP2)
hold. In other words, the isolated invariant set~$S$, and the
conditions~(IS1), (IS2), and~(IP3) are not needed for the above.
Thus, the index map $I_P: H_*(P_1,P_2)\to H_*(P_1,P_2)$ can
still be defined as in Definition~\ref{def:indexmap}.
\end{rem}
%
%
\subsection{Normal functors}
\label{subsec:normalfunctor}
Next we need to recall some definitions and results from category
theory, in particular centered around the notion of normal functors.
For this, let~$\cE$ denote a category. We define the category of
endomorphisms of~$\cE$, denoted by~$\Endo(\cE)$ as follows:
\begin{itemize}
\item The objects of~$\Endo(\cE)$ are pairs~$(A,a)$, where
$A\in\cE$ and $a\in\cE(A,A)$ is an endomorphism of~$A$.
\item The set of morphisms from~$(A,a)\in \Endo(\cE$) to
$(B,b)\in \Endo(\cE)$ is the subset of~$\cE(A,B)$ consisting
of exactly those morphisms $\varphi\in\cE(A,B)$ for which
$\varphi a = b\varphi$.
\end{itemize}
We write $\varphi:(A,a)\rightarrow (B,b)$ to denote that~$\varphi$
is a morphism from~$(A,a)$ to~$(B,b)$ in~$\Endo(\cE)$. It is easy
to see that if $\varphi: (A,a)\to (B,b)$ is a morphism in~$\Endo(\cE)$
which is an isomorphism in~$\cE$, then it is also an isomorphism
in~$\Endo(\cE)$. Note that any endomorphism $a\in \cE(A,A)$ is in
particular a morphism $a:(A,a)\map (A,a)$ in~$\Endo(\cE)$. Such
morphisms of~$\Endo(\cE)$ are called \textit{induced}.

Now let $L:\Endo(\cE)\map\cC$ be a functor. We say that~$L$ is
{\em normal\/} if~$L(a)$ is an isomorphism in~$\cC$ for every
induced morphism $a:(A,a)\map (A,a)$ in~$\Endo(\cE)$. Then we have the 
following result.
\begin{prop}[Isomorphism inducing property of normal functors]
\label{prop:normal}
In the situation above, let
$L:\Endo(\cE)\map\cC$ denote a normal functor, and let
$\varphi: A\to B$ and $\psi:B\to A$ be morphisms in~$\cE$.
Then $\varphi: (A, \psi \varphi)\to (B,\varphi \psi)$ is a
morphism in the category~$\Endo(\cE)$, and~$L(\varphi)$ is
an isomorphism in~$\cC$.
\end{prop}
\begin{proof}
Clearly we have that $\varphi$ is a morphism from $(A, \psi \varphi)$ to $(B, \varphi \psi)$ in $\Endo(\cE)$
and $\psi$ is a morphism from $(B, \varphi \psi)$ to $(A, \psi \varphi)$.
In addition, one obtains the commutative diagram
\[
\begin{diagram}
 \node{(B, \varphi \psi)}
 \arrow{s,l}{\psi}
 \arrow{e,t}{\varphi \psi}
 \node{(B, \varphi \psi)}
 \arrow{s,r}{\psi}
 \\
 \node{(A,\psi \varphi)}
 \arrow{e,t}{\psi \varphi}
 \arrow{ne,t}{\varphi}
 \node{(A,\psi \varphi)}
\end{diagram}
\]
If we now apply the functor~$L$ to this diagram, then the horizontal
morphisms become isomorphisms in~$\cC$. Thus, the image~$L(\varphi)$ has
both a left and a right inverse, and therefore it is also an isomorphism.
\qed
\end{proof}

\medskip
We would like to point out that if~$W$ denotes the class of induced
morphisms in~$\Endo(\cE)$, then the natural functor $\Endo(\cE)\to
\Endo(\cE)[W^{-1}]$ to the localization is universal in the sense
that any other normal functor $\Endo(\cE)\to \cC$ factorizes through
it, see also~\cite{szymczak:95a}. We close this section
with one specific example of a normal functor. For further examples
we refer the reader to the paper~\cite{mrozek:92a}.
\begin{ex}[The Leray functor]
\label{ex:lerayfunctor}
{\em
For the example computations of this paper, we make use of the 
specific normal functor introduced in~\cite{mrozek:90a}, the
{\it Leray functor\/}. For this, let~$\Mod$ denote the category
of graded moduli over the ring~$R$ together with homomorphisms of
degree zero. Using the setting for the definition of normal functors
from above, we consider the categories
\begin{displaymath}
  \cE = \Mod
  \quad\mbox{ and }\quad
  \cC = \Auto(\Mod) \; ,
\end{displaymath}
where $\Auto(\Mod) \subset \Endo(\Mod)$ is the subcategory of
automorphisms of~$\Mod$. Then the {\it Leray functor\/} 
$\Leray : \Endo(\Mod) \map \Auto(\Mod)$ can be defined as 
the composition of the following maps:
\begin{itemize}
\item Let $(H,h) \in \Endo(\Mod)$ be arbitrary. Then the {\it generalized
kernel\/} of~$h$ can be defined as
\begin{displaymath}
  \mathrm{gker}(h) := \bigcup_{n \in \NN} h^{-n}(0) \; ,
\end{displaymath}
and one can easily see that the map~$h : H \to H$ induces a well-defined
map $h' : H/\mathrm{gker}(h) \to H/\mathrm{gker}(h)$. Thus, the definition
\begin{displaymath}
  L'(H,h) := (H/\mathrm{gker}(h), h') \in \Mono(\Mod) \subset \Endo(\Mod)
\end{displaymath}
gives an object in the category~$\Mono(\Mod)$ of monomorphisms of~$\Mod$.
Furthermore, it is straightforward to define~$L'(\phi)$ also for morphisms~$\phi$
in~$\Endo(\Mod)$, and to show that in this way one obtains a well-defined
contravariant functor $L' : \Endo(\Mod) \to \Mono(\Mod)$.
\item Now let $(H,h) \in \Mono(\Mod)$ be arbitrary. Then the {\it generalized
image\/} of~$h$ can be defined as
\begin{displaymath}
  \mathrm{gim}(h) := \bigcap_{n \in \NN} h^n(H) \; ,
\end{displaymath}
and it is not difficult to verify that the map~$h : H \to H$ induces a well-defined
map $h'' : \mathrm{gim}(h) \to \mathrm{gim}(h)$. Thus, the definition
\begin{displaymath}
  L''(H,h) := (\mathrm{gim}(h), h'') \in \Auto(\Mod) \subset \Endo(\Mod)
\end{displaymath}
gives an object in the category~$\Auto(\Mod)$ of automorphisms of~$\Mod$.
In addition, it is again straightforward to define~$L''(\phi)$ also for
morphisms~$\phi$ in~$\Mono(\Mod)$, and to show that this time one obtains
a well-defined contravariant functor $L'' : \Mono(\Mod) \to \Auto(\Mod)$.
\item Finally, the Leray functor is defined as $\Leray := L'' \circ L'$.
\end{itemize}
For more details on the above construction, as well as the proof that
the Leray functor is indeed a normal functor, we refer the reader
to~\cite[Section~4]{mrozek:90a}. For our applications below, we note
that by the construction of~$\Leray$ we have the implication
\begin{equation} \label{eqn:lerayid}
  (H,h) \in \Auto(\Mod) \subset \Endo(\Mod)
  \quad\Longrightarrow\quad
  \Leray(H,h) = (H,h) \; ,
\end{equation}
i.e., the Leray functor is the identity on~$\Auto(\Mod) \subset
\Endo(\Mod)$. This fact will enable us to determine the Conley 
index of isolated invariant sets in many situations.
}
\end{ex}
%
%
%
\subsection{The Conley index}
\label{subsec:conleyindex}
After these preparations we can finally define the Conley index. A first
attempt would be to use the index map $I_P : H_*(P) \to H_*(P)$ introduced
in Definition~\ref{def:indexmap}. Unfortunately, however, this would mean
that the index depends on the chosen index pair of the isolated
invariant set.

This issue can be addressed by using the concept of normal functors
from the last section. More precisely, let~$\Mod$ denote as before the category
of graded moduli over the ring~$R$ and let $L:\Endo(\Mod)\to\Auto(\Mod)$
be a fixed normal functor.  Note that if~$P$ is an index pair for an
isolated invariant set~$S$ in an isolating set~$N$, then one obtains
$(H_*(P),I_P) \in \Endo(\Mod)$. Thus, the \textit{$L$-reduction} $L(H_*(P),I_P)$ is an automorphism of a graded module over~$R$, and we have the
following crucial result.
\begin{thm}[Well-definedness of the Conley index]
\label{th_ind}
In the situation described above, the isomorphism type of $L(H_*(P),I_P)
\in \Auto(\Mod)$ does not depend on the choice of the isolating set~$N$
for the isolated invariant set~$S$, or on the chosen index pair~$P$ in~$N$.
\end{thm}
\begin{proof}
To begin, let~$M$ and~$N$ be two isolating sets for~$S$, and let~$P$ and~$Q$
denote two index pairs in~$N$ and~$M$, respectively. Our goal is to establish
the equivalence $L(H_*(P),I_P) \cong L(H_*(Q),I_Q)$. This is accomplished in
five steps.

{\em Step 1.} We first consider the special case
\begin{itemize} \setlength{\itemsep}{3pt}
\item[(i)] $M=N$,
\item[(ii)] $P\subset Q$,
\item[(iii)] $P_1=Q_1$ or $P_2=Q_2$,
\item[(iv)] $F(Q)\cap N\subset P$.
\end{itemize}
Let $D=(D_1,D_2)$ be the pair of closed sets defined by $D_i=P_i\cup
\cl(F(Q_1)\setminus N)$ for~$i=1,2$. By~(iv) we may treat~$F$ as a map
of pairs $F_{QD} = F|_Q^{D}: Q\mto D$. In view of~(i) and~(ii), we also
have $\overline{P}\subset D\subset \overline{Q}$. This gives the
following commutative diagram 
\[
\begin{diagram}
 \node{P}
 \arrow[2]{s,l}{j}
 \node{\overline{P}}
\arrow{w,t,dd}{F_P}
 \arrow{s,r}{k}
 \node{P}
 \arrow{w,t}{\iota_P}
 \arrow[2]{s,r}{j}
 \\
\node{}
\node{D}
\arrow{sw,l,dd}{F_{QD}}
\arrow{s,r}{l}
\node{}
\\
 \node{Q}
 \node{\overline{Q}}
\arrow{w,t,dd}{F_Q}
 \node{Q}
 \arrow{w,t}{\iota_Q}
\end{diagram}
\]
in which vertical arrows denote inclusions. Since $F$ induces a
map $F_{PD}=F|_P^D$, by Lemmas~\ref{lemma1} and~\ref{lemma2} we
have $k_*(F_P)_*=(F_{PD})_*=(F_{QD})_*j_*$ and $l_*(F_{QD})_*=(F_Q)_*$.
We then obtain a commutative diagram
\[
\begin{diagram}
\dgARROWLENGTH=1.8em
 \node{H_*(P)}
 \arrow[2]{s,l}{j_*}
 \arrow{e,t}{(F_P)_*}
 \node{H_*(\overline{P})}
 \arrow{s,r}{k_*}
 \node{H_*(P)}
 \arrow{w,t}{(\iota_P)_*}
 \arrow[2]{s,r}{j_*}
 \\
\node{}
\node{H_*(D)}
\arrow{s,r}{l_*}
\node{}
\\
 \node{H_*(Q)}
 \arrow{e,t}{(F_Q)_*}
 \arrow{ne,l}{(F_{QD})_*}
 \node{H_*(\overline{Q})}
 \node{H_*(Q)}
 \arrow{w,t}{(\iota_Q)_*}
\end{diagram}
\]
%
%
We claim that $k$ induces an isomorphism in homology. Indeed, if $P_1=Q_1$, $k$ is the identity. Otherwise, by (iii) we have $P_2=Q_2$. In this case we claim that $k$ fulfills the hypothesis of strong excision, namely, $$P_1\setminus (P_2 \cup \cl(F(P_1)\setminus N))= P_1\setminus (P_2 \cup \cl(F(Q_1)\setminus N)).$$
Inclusion of the second subspace in the first is trivial, and their difference is $$P_1\cap \cl(F(Q_1)\setminus N) \setminus (P_2 \cup \cl(F(P_1)\setminus N)) \subset Q_1\cap \cl(F(Q_1)\setminus N) \setminus P_2,$$ which is equal to $Q_1\cap \cl(F(Q_1)\setminus N) \setminus Q_2$, and this is empty by (IP2).    

If one defines~$I_{QP}:=(\iota_P)_*^{-1}k_*^{-1} (F_{QD})_*$,
then we get the commutative diagram in~$\Mod$ given by
\[
\begin{diagram}
 \node{H_*(P)}
 \arrow{s,l}{j_*}
 \arrow{e,t}{I_P}
 \node{H_*(P)}
 \arrow{s,r}{j_*}
 \\
 \node{H_*(Q)}
 \arrow{e,t}{I_Q}
 \arrow{ne,l}{I_{QP}}
 \node{H_*(Q)}
\end{diagram}\quad,
\]
and $L(j_*):L(H_*(P), I_{QP}j_*)=L(H_*(P), I_P)\to L(H_*(Q), j_*I_{QP})=L(H_*(Q),I_Q)$
is an isomorphism in view of Proposition~\ref{prop:normal}.

{\em Step 2.} Next we drop assumption~(iv). According to Lemma~\ref{lem:seq-Qi}
we can find a sequence $Q^0,Q^1, \ldots ,Q^n$ of index pairs such that $Q^0=Q$ and
$Q^n=P$, and such that each pair $(Q^{k+1}$, $Q^k)$ satisfies assumptions (i)--(iv). Due to {\em Step~1\/} the $L$-reductions $L(H_*(Q^{k},I_{Q^k}))$ and $L(H_*(Q^{k+1},I_{Q^{k+1}}))$ are isomorphic, and the conclusion follows.

{\em Step 3}. We now drop assumptions~(iii) and~(iv). For this, notice
that in view of Lemma~\ref{inter2} the pairs $R=(P_1,P_1\cap Q_2)$ and
$T=(P_1\cup Q_2,Q_2)$ are index pairs. 

The pairs~$P$ and~$R$ satisfy assumptions~(ii) and~(iii), and therefore they
have isomorphic $L$-reductions. The same holds for~$T$ and~$Q$. On the other
hand, the inclusion $j:R \hookrightarrow T$ induces an isomorphism
$j_*:H_*(R)\to H_*(T)$ by strong excision. Since $R\subset T$, we have an
inclusion $\overline{j}:\overline{R}\hookrightarrow \overline{T}$, as well
as the commutative diagram
\[
\begin{diagram}
 \node{H_*(R)}
 \arrow{s,l}{j_*}
 \arrow{e,t}{(F_R)_*}
 \node{H_*(\overline{R})}
 \arrow{s,r}{\overline{j}_*}
 \node{H_*(R)}
 \arrow{w,t}{(\iota_R)_*}
 \arrow{s,r}{j_*}
 \\
 \node{H_*(T)}
 \arrow{e,t}{(F_T)_*}
 \node{H_*(\overline{T})}
 \node{H_*(T)}
 \arrow{w,t}{(\iota_T)_*}
\end{diagram}
\]
Thus, one obtains $j_*I_R = j_*(\iota_R)_*^{-1}(F_R)_* = 
(\iota_T)_*^{-1}\overline{j}_*(F_R)_* = (\iota_T)_*^{-1}(F_T)_*j_* =
I_Tj_*$. This shows that $j_*\in \Endo(\Mod)((H_*(R),I_R),(H_*(T),I_T))$,
and since~$j_*$ is an isomorphism in~$\Mod$, it also is an isomorphism
in~$\Endo(\Mod)$. This in turn implies that $L(j_*):L(H_*(R),I_R)\to
L(H_*(T),I_T)$ is an isomorphism, and that~$P$ and~$Q$ indeed have
isomorphic $L$-reductions.

{\em Step 4}. Now we only assume~(i). By Lemma~\ref{inter1} the pair
$P\cap Q$ is an index pair. Hence, the claim follows from {\em Step 3}
applied to~$P\cap Q\subset P$ and~$P\cap Q\subset Q$.

{\em Step 5}. Finally, we drop all auxiliary assumptions. We have already
proved that the isomorphism type of the $L$-reduction depends only on the
isolating set for~$S$. Moreover, since by Proposition~\ref{prop:iso-set-intersection},
the intersection of two isolating sets is again an isolating set, we
may assume $M\subset N$. 

Consider the index pairs~$P^M$ for~$S$ in~$M$ and~$P^N$ for~$S$ in~$N$.
In view of Proposition~\ref{prop:ip-M-N} and Proposition~\ref{prop:barip-M-N}
we then have the commutative diagram
\[
\begin{diagram}
\node{P^M}
\arrow{s,l}{j}
 \node{\overline{P^M}}
\arrow{w,t,dd}{F_{P^M}}
 \arrow{s,r}{k}
 \node{P^M}
 \arrow{w,t}{\iota_{P^M}}
 \arrow{s,r}{j}
 \\
 \node{P^N}
 \node{\overline{P^N}}
\arrow{w,t,dd}{F_{P^N}} 
\node{P^N}
 \arrow{w,t}{\iota_{P^N}}
\end{diagram}
\]
in which vertical arrows denote inclusions. Then $I_{P^N}j_*=j_* I_{P^M}$,
which implies that $j_*:(H_*(P^M), I_{P^M})\to (H_*(P^N), I_{P^N})$ is a
morphism in~$\Endo(\Mod)$. On the other hand, since~$P^M$ and~$P^N$ are
saturated by Theorem~\ref{thm:ip-existence}, strong excision shows that
$j_*:H_*(P^M)\to H_*(P^N)$ is an isomorphism in~$\Mod$. Thus, the map~$j_*$
is an isomorphism in~$\Endo(\Mod)$, and then so is~$L(j_*)$.
\qed
\end{proof}

\medskip
Based on the above result, the Conley index can now be defined as follows.
We would like to point out that the functor~$L$ in the definition could
be, for example, the computationally convenient Leray functor of Example \ref{ex:lerayfunctor}.
\begin{defn}[The Conley index]
\label{def:con}
The $L$-reduction $L(H_*(P),I_P)$ will be called the {\em homological
Conley index of~$S$}, and be denoted by~$C(S,F)$, or simply~$C(S)$
if~$F$ is clear from context. Due to Theorem~\ref{th_ind} the Conley
index $C(S)\in \Auto (\Mod)$ is well-defined up to isomorphism.
\end{defn}
In order to illustrate the above abstract definition of the Conley index,
we now briefly return to our earlier two examples and determine the
Conley indices of all the Morse sets shown in Figures~\ref{fig:simplecvf}
and~\ref{fig:reflectedcvf}.
\begin{ex}[Sample Conley index computations]
\label{ex:conleyindices}
{\em
We return one last time to the two simple multivalued maps~$F : X \mto X$
and~$G : X \mto X$ from Examples~\ref{ex:simplecvf} and~\ref{ex:reflectedcvf},
respectively. We have already seen that these maps give rise to associated
Morse decompositions with three and five isolated invariant sets, which
themselves are subsets of the finite topological space $X = \{ A, B, C,
AB, AC, BC, ABC \}$. Notice that in view of Example~\ref{ex:extendedtoppairs}
in all of these cases the extended topological pair~$\overline{P}$ equals the
index pair~$P$ that was chosen for each isolated invariant set. Thus, the
index map~$I_P$ is simply given by $I_P = (F_P)_* : H_*(P) \to H_*(P)$ for
the sets in~(\ref{ex:standardindexpairs1}), and similarly for the isolated
invariant sets in~(\ref{ex:standardindexpairs2}).

Consider now the multivalued map~$F : X \mto X$ from Example~\ref{ex:simplecvf}.
For the sake of simplicity, we compute the Conley index for the ring $R = \ZZ$
and with respect to the Leray functor. Then for the isolated invariant
set~$S_1 = \{ A, B, C \}$ one can easily see that $H_0(P^{S_1,N_1}) \simeq \ZZ^3$.
Moreover, the index map~$I_{P^{S_1,N_1}}$ maps the generators in a cyclic fashion,
i.e., it is an automorphism. Based on~(\ref{eqn:lerayid}), this shows that the
Conley index with respect to~$\Leray$ is just $(H_*(P^{S_1,N_1}), I_{P^{S_1,N_1}})$.
In a similar way, one can determine the Conley index for all the isolated invariant
sets in Figure~\ref{fig:simplecvf} as
\begin{displaymath}
  \begin{array}{lclcl}
    \DS S_1 = \{ A, B, C \} & : &
      \DS H_0(P^{S_1,N_1}) \simeq \ZZ^3 & \mbox{ with } &
      \DS I_{P^{S_1,N_1}} (e_i) = e_{(i+1) \;\mathrm{mod}\; 3}, \\
    \DS S_2 = \{ AB, BC, AC \} & : &
      \DS H_1(P^{S_2,N_2}) \simeq \ZZ^3 & \mbox{ with } &
      \DS I_{P^{S_2,N_2}} (e_i) = e_{(i+1) \;\mathrm{mod}\; 3}, \\
    \DS S_3 = \{ ABC \} & : &
      \DS H_2(P^{S_3,N_3}) \simeq \ZZ & \mbox{ with } &
      \DS I_{P^{S_3,N_3}} (e_i) = e_i,
  \end{array}
\end{displaymath}
where in each case all unlisted homology groups are trivial, and the listed
group~$\ZZ^k$ has a suitable basis~$\{ e_0, e_1, \ldots, e_{k-1} \}$. Similarly,
for the multivalued map~$G$ from Example~\ref{ex:reflectedcvf} and the
isolated invariant sets in Figure~\ref{fig:reflectedcvf} one obtains
\begin{displaymath}
  \begin{array}{lclcl}
    \DS R_1 = \{ A \} & : &
      \DS H_0(P^{R_1,M_1}) \simeq \ZZ & \mbox{ with } &
      \DS I_{P^{R_1,M_1}} (e_i) = e_i, \\
    \DS R_2 = \{ B, C \} & : &
      \DS H_0(P^{R_2,M_2}) \simeq \ZZ^2 & \mbox{ with } &
      \DS I_{P^{R_2,M_2}} (e_i) = e_{(i+1) \;\mathrm{mod}\; 2}, \\
    \DS R_3 = \{ BC \} & : &
      \DS H_1(P^{R_3,M_3}) \simeq \ZZ & \mbox{ with } &
      \DS I_{P^{R_3,M_3}} (e_i) = -e_i, \\
    \DS R_4 = \{ AB, AC \} & : &
      \DS H_1(P^{R_4,M_4}) \simeq \ZZ^2 & \mbox{ with } &
      \DS I_{P^{R_4,M_4}} (e_i) = e_{(i+1) \;\mathrm{mod}\; 2}, \\
    \DS R_5 = \{ ABC \} & : &
      \DS H_2(P^{R_5,M_5}) \simeq \ZZ & \mbox{ with } &
      \DS I_{P^{R_5,M_5}} (e_i) = -e_i,
  \end{array}
\end{displaymath}
where we use the same conventions as above. We leave the details of
these straightforward computations to the reader.
}
\end{ex}
%
%
%
\section{Properties of the Conley index}
\label{sec:conleyindexprop}
In this section, we present first properties of the Conley index for
multivalued maps defined in the last section. In addition to the
Wa\.zewski property, we also briefly address continuation.
%
%
\subsection{The Wa\.zewski property}
In classical Conley theory, the Wa\.zewski property is central, as it
allows one to deduce the existence of a nontrivial isolated invariant
set~$S$ from a nontrivial index, and the latter can be computed from an
index pair without explicit knowledge of~$S$.

In order to show that the same result still holds in the multivalued
context of the present paper, let $P=(P_1,P_2)$ denote a topological
pair of closed subspaces of~$X$. Suppose further that $N=P_1$ satisfies
conditions~(IP1) and~(IP2), i.e., we have the inclusion
$P_1\cap (\cl(F(P_1)\setminus P_1) \cup F(P_2))\subset P_2$.
In view of Remark~\ref{index-mapnew}, the index map $I_P:H_*(P)\to H_*(P)$
is defined in this situation. Then we have the following result.
\begin{prop}[Wa\.zewski property]
Suppose that~$X$ is a finite~$T_0$ topological space and that the
multivalued map $F:X\mto X$ is lower semicontinuous with closed and
acyclic values. Moreover, let $P=(P_1,P_2)$ be a pair of closed subspaces
of~$X$ such that
\begin{displaymath}
  P_1 \cap \left( \cl \left( F(P_1) \setminus P_1 \right)
  \cup F(P_2) \right) \; \subset \; P_2 .
\end{displaymath}
If one further has $L(H_*(P),I_P) \neq 0 \in \Auto(\Mod)$,
then $\Inv(P_1\setminus P_2)\neq \emptyset$. 
\end{prop}
\begin{proof}
Suppose $\Inv(P_1\setminus P_2)=\emptyset$. Then $N=P_1$ is an isolating
set for the invariant set $S=\emptyset$, and~$P$ is an index pair for~$S$
in~$N$. According to our hypothesis, we have $C(S)\neq 0$. But this is
absurd since~$S$ admits $N'=\emptyset$ as isolating set and
$P'=(\emptyset,\emptyset)$ is an index pair for~$S$ in~$N'$.
Thus, we have the equality $H_*(P')=0$, as well as
$C(S)=L(H_*(P'),I_{P'})=0$.
\qed
\end{proof}
%
%
\subsection{Homotopies and continuation}
\label{sec:continuation}
As our second property of the Conley index we address the fundamental
concept of continuation. For this, we first need to review some results
on homotopies in finite topological spaces.

Let~$X$ and~$Y$ be two finite~$T_0$ spaces. Two lower semicontinuous
multivalued maps $F,G:X\multimap Y$ with closed and acyclic values are
called \textit{homotopic} if there exists a lower semicontinuous map
$H:X\times [0,1] \multimap Y$ with closed and acyclic values such that
$H(x,0)=F(x)$ and $H(x,1)=G(x)$ for every $x\in X$. This definition
extends in a natural way to maps $(X,A)\multimap (Y,B)$ between pairs
of finite~$T_0$ spaces by requiring that $H:X\times [0,1]\to Y$
maps~$(a,t)$ to $H(a,t)\subset B$ for every $a\in A$ and $t\in [0,1]$. 

General homotopies in the setting of finite topological spaces can be
more succinctly described as follows. Define an order on the set of all
lower semicontinuous multivalued maps $X \multimap Y$ with closed and
acyclic values by letting~$F\le G$ if we have $F(x)\subset G(x)$ for all
$x\in X$. A sequence $F=F_0\le F_1 \ge F_2\le \ldots F_k=G$ is called a
\textit{fence} from~$F$ to~$G$. Then the proof of the following result
is essentially the same as the proof of~\cite[Proposition~8.1]{barmak:etal:20a},
and therefore we omit it.
\begin{prop}[Homotopy characterization via fences]
Let~$X$ and~$Y$ be two finite~$T_0$ spaces and let $F,G:X\multimap Y$
be two lower semicontinuous multivalued maps with closed and acyclic
values. Then the maps~$F$ and~$G$ are homotopic if and only if there
exists a fence $F=F_0\le F_1\ge F_2\le \ldots F_k=G$ of lower
semicontinuous multivalued maps $X\multimap Y$ with closed and
acyclic values.

Furthermore, if the maps~$F,G:(X,A)\multimap (Y,B)$
are maps of pairs of finite~$T_0$ spaces, then they are homotopic
if and only if there exists a fence as above in which the maps are
maps of pairs $(X,A)\multimap (Y,B)$.
\end{prop}
In terms of the associated maps in homology we have the following
result, which is in the spirit of~\cite[Corollary~8.2]{barmak:etal:20a}.
\begin{lem}[Homotopic maps induce the same map in homology]
\label{comparable}
Let~$X,Y$ be finite~$T_0$ spaces, and let $F,G:X\multimap Y$ be two
homotopic lower semicontinuous multivalued maps with closed and acyclic
values. Then $F_*=G_*:H_*(X)\to H_*(Y)$ for the maps induced in homology.
The same result holds more generally for pairs.
\end{lem}
\begin{proof}
We may assume $F\le G$. Consider the following commutative diagram
\[
\begin{diagram}
\node{}
\node{F}
\arrow[2]{s,r}{j}
\arrow{sw,l}{p_1}
\arrow{se,l}{p_2}
\node{}
\\
\node{X}
\node{}
\node{Y}
\\
\node{}
\node{G}
\arrow{nw,r}{\widetilde{p}_1}
\arrow{ne,r}{\widetilde{p}_2}
\node{}
\end{diagram}
\]
in which~$j$ denotes the inclusion between the graphs, and the other maps
are the projections to the first or second coordinate. Since~$p_1$
and~$\widetilde{p}_1$ induce isomorphisms in homology, so does~$j$.
This immediately implies
\begin{displaymath}
  G_*=(\widetilde{p}_2)_*(\widetilde{p}_1)_*^{-1}=
  (p_2)_*(j_*)^{-1}j_*(p_1)^{-1}_*=F_*
  \; : \; H_*(X)\to H_*(Y) .
\end{displaymath}
The result for pairs follows with the exact same proof.
\qed
\end{proof}

\medskip
The following definition introduces the notion of {\em continuation\/} for the
setting of multivalued maps in finite topological spaces.
\begin{defn}[Continuation of isolated invariant sets]
\label{def:continuation}
%
%
Let~$X$ be a finite~$T_0$ space and let $F,G:X\multimap X$ be two lower
semicontinuous multivalued maps with closed and acyclic values such that
$F\le G$ or $F\ge G$. Moreover, let $S_F,S_G\subset X$ be isolated
invariant sets for~$F$ and~$G$, respectively. We say that~$(S_F,F)$
and~$(S_G,G)$ (or just~$S_F$ and~$S_G$) are {\em related by an elementary
continuation} if there exist isolating sets~$N_F$ and~$N_G$ for~$S_F$
and~$S_G$ with respect to~$F$ and~$G$, respectively, as well as a pair
$P=(P_1,P_2)$ which is both
\begin{itemize}
\item an index pair for~$S_F$ in~$N_F$ with respect to~$F$, and
\item an index pair for~$S_G$ in~$N_G$ with respect to~$G$.
\end{itemize}
More generally, let $F,G:X\multimap X$ denote two homotopic lower
semicontinuous multivalued maps with closed and acyclic values. We
say that isolated invariant sets~$S_F$ and~$S_G$ for~$F$ and~$G$,
respectively, are {\em related by continuation}, if there exists
a fence $F=F_0\le F_1\ge F_2\le \ldots F_k=G$ of lower semicontinuous
multivalued maps $X\multimap X$ with closed and acyclic values, as
well as isolated invariant sets~$S_i$ for~$F_i$, for $0\le i\le k$,
such that $S_0=S_F$, $S_k=S_G$, and $(S_i,F_i)$, $(S_{i+1},F_{i+1})$
are related by an elementary continuation for each $0\le i <k$.
%
\end{defn}
As in the classical case, we then have the following central result.
\begin{prop}[Continuation]
\label{prop:continuation}
Let $F,G:X\multimap X$ be homotopic lower semicontinuous multivalued
maps with closed and acyclic values, and let~$S_F$ and~$S_G$ be 
isolated invariant sets for~$F$ and~$G$, respectively, which are
related by continuation. Then the Conley index~$C(S_F,F)$ is
isomorphic to the Conley index~$C(S_G,G)$.
\end{prop}
\begin{proof}
We can assume without loss of generality that $F\le G$, and that~$S_F$
and~$S_G$ are related by an elementary continuation. Let~$N_F$, $N_G$,
and~$P$ be as in Definition~\ref{def:continuation}. Since we have
$F\le G$, one obtains the inclusion
\begin{eqnarray*}
  \overline{P_i}^F & = &
    P_i\cup \cl(F(P_1)\setminus N_F) \; = \;
    P_i\cup \cl(F(P_1)\setminus P_1) \\[0.5ex]
  & \subset &
    P_i\cup \cl(G(P_1)\setminus P_1) \; = \;
    P_i\cup \cl(G(P_1)\setminus N_G) \; = \;
    \overline{P_i}^G.
\end{eqnarray*}
Thus we have a (non-commutative) diagram
\[
\begin{diagram}
\node{}
\node{\overline{P}^F}
\arrow{sw,t,dd}{F_{P}}
\arrow[2]{s,r}{j}
\node{}
\\
\node{P}
\node{}
\node{P}
\arrow{nw,l}{\iota_{P,F}}
\arrow{sw,r}{\iota_{P,G}}
\\
\node{}
\node{\overline{P}^G}
\arrow{nw,r,dd}{G_P}
\node{}
\end{diagram}
\]
in which~$j$ denotes inclusion. According to Lemma~\ref{lemma2}
one has $j_*(F_P)_*=(jF_P)_*$ as a map from~$H_*(P)$ to~$H_*(\overline{P}^G)$.
Moreover, our assumption $F \le G$ immediately implies $jF_P\le G_P$, and 
therefore Lemma~\ref{comparable} yields $(jF_P)_*=(G_P)_*$. Since the right
triangle is in fact commutative, the map~$j_*$ is an isomorphism. Thus the
index map~$I_{P,F}$ of~$P$ with respect to~$F$ is given by
\begin{displaymath}
  (\iota_{P,F})^{-1}_* (F_P)_* \; = \;
  (\iota_{P,F})^{-1}_*j_*^{-1} j_*(F_P)_* \; = \;
  (\iota_{P,G})_*^{-1}(G_P)_* \; = \;
  I_{P,G},
\end{displaymath}
and this furnishes in particular $L(H_*(P), I_{P,F})=L(H_*(P), I_{P,G})$.
In other words, the Conley indices~$C(S_F,F)$ and~$C(S_G,G)$ are isomorphic.
\qed
\end{proof}

\medskip
To close this section, we present a detailed example which illustrates the 
concept of continuation, and also provides further insight into isolated
invariant sets and their Conley indices.
\begin{figure} \centering
  \setlength{\unitlength}{1 cm}
  \begin{picture}(12.5,4.0)
    \put(0.0,0.0){
      \includegraphics[height=4.0cm]{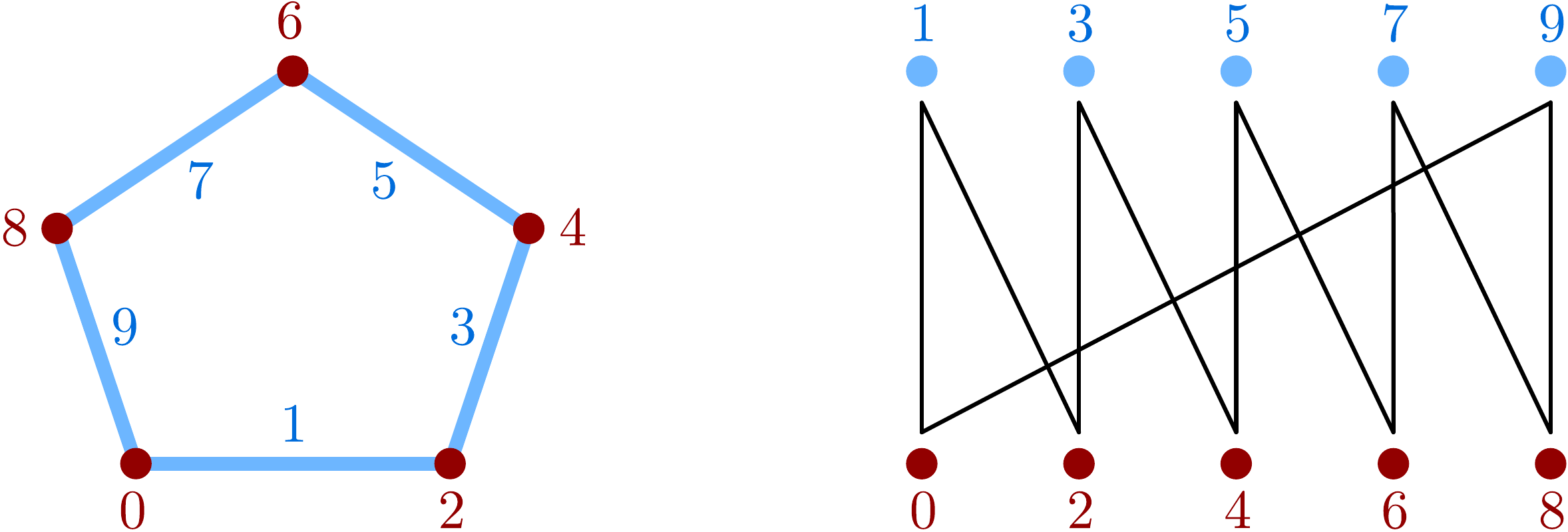}} 
  \end{picture}
  \caption{The finite $T_0$ topological space~$X$ used in Example~\ref{ex:continuation}.
           The left panel shows a simplicial complex in the form of a pentagon,
           given by five vertices and five edges. Using the order given by the
           face relationship, one obtains the ten-point finite topological
           space~$X$, which is shown in the right panel via its poset
           representation.}
  \label{fig:continuationX}
\end{figure}%
\begin{ex}[Continuation of isolated invariant sets]
\label{ex:continuation}
{\em
For this example, we let~$X$ denote the finite topological space which
is generated by a simplicial representation of a pentagon, as shown
in Figure~\ref{fig:continuationX}. Using the Alexandrov topology induced
by the face relation, one obtains the ten-point topological space~$X$
indicated in the right panel of the figure as a poset. Note that we
can identify~$X$ with the set~$\ZZ_{10}$, where the topology is given
as in the poset.
\begin{figure} \centering
  \setlength{\unitlength}{1 cm}
  \begin{picture}(12.0,7.0)
    \put(0.0,0.0){
      \includegraphics[height=7.0cm]{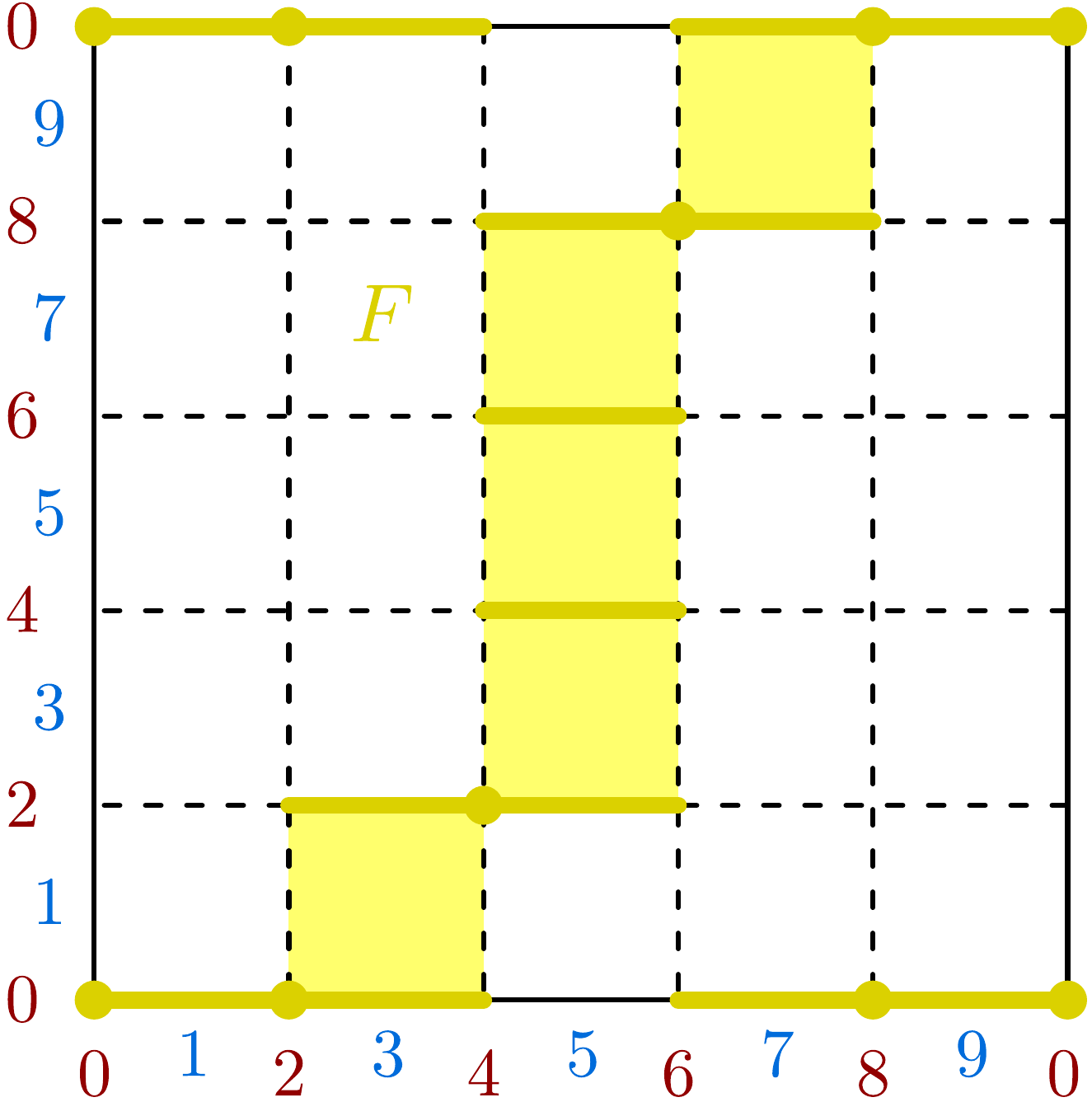}}  
    \put(8.5,1.0){\makebox(3.5,5.5){
      \begin{tabular}{|c||c|} \hline
        $x$ & $F(x)$ \\[0.5ex] \hline
        0 & 0 \\
        1 & 0 \\
        2 & 0 \\
        3 & 0 1 2 \\
        4 & 2 \\
        5 & 2 3 4 5 6 7 8 \\
        6 & 8 \\
        7 & 0 8 9 \\
        8 & 0 \\
        9 & 0 \\ \hline
      \end{tabular}
      }} 
  \end{picture}
  \caption{Definition of the multivalued map $F : X \mto X$. The
           left image shows the graph of~$F$. For this, we represent
           the pentagon from Figure~\ref{fig:continuationX} as a line
           segment, whose end points are identified. The table on the
           right lists all function values~$F(x)$.}
  \label{fig:continuationF}
\end{figure}%
\begin{figure} \centering
  \setlength{\unitlength}{1 cm}
  \begin{picture}(12.0,7.0)
    \put(0.0,0.0){
      \includegraphics[height=7.0cm]{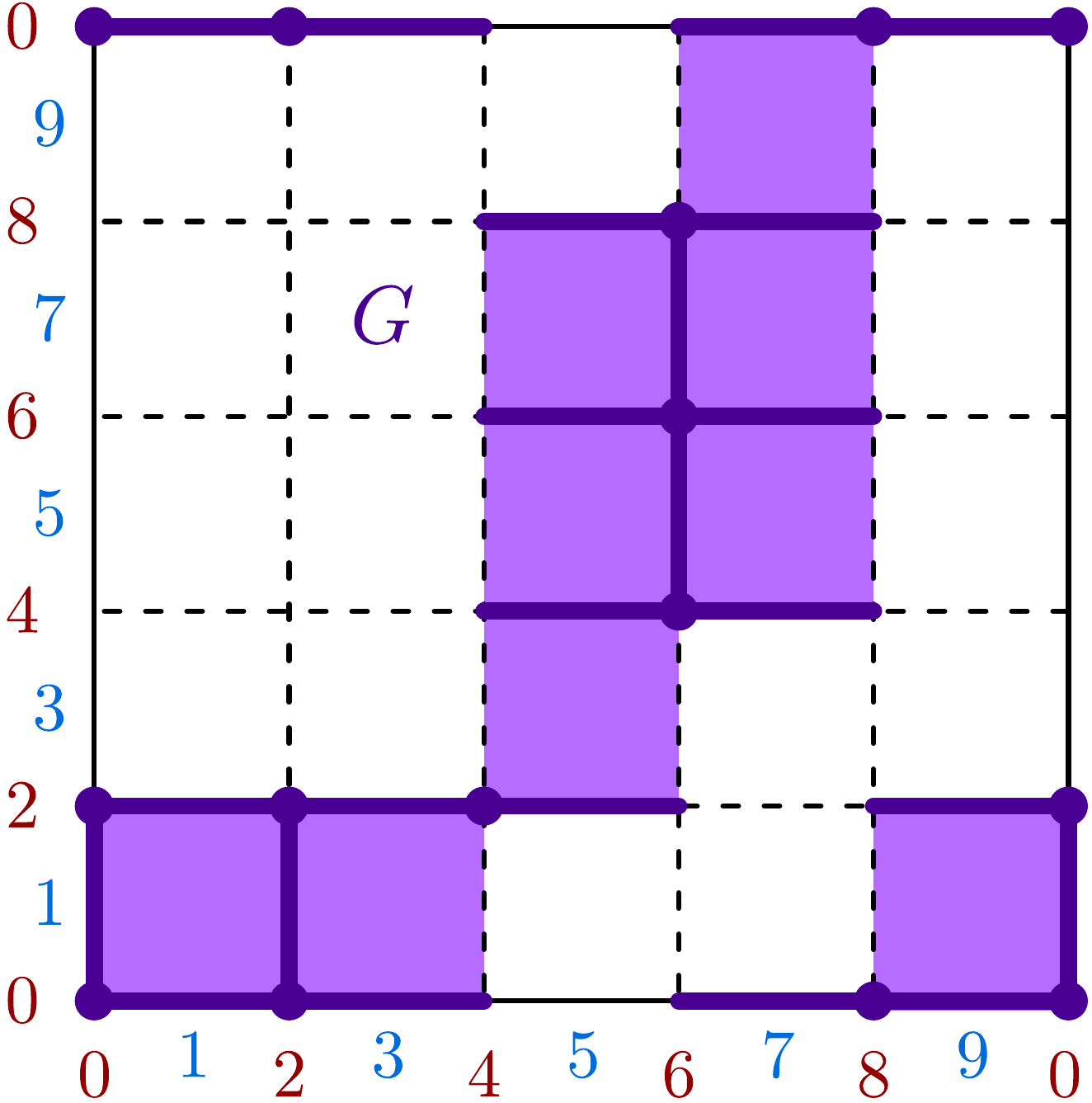}}  
    \put(8.5,1.0){\makebox(3.5,5.5){
      \begin{tabular}{|c||c|} \hline
        $x$ & $G(x)$ \\[0.5ex] \hline
        0 & 0 1 2 \\
        1 & 0 1 2 \\
        2 & 0 1 2 \\
        3 & 0 1 2 \\
        4 & 2 \\
        5 & 2 3 4 5 6 7 8 \\
        6 & 4 5 6 7 8 \\
        7 & 0 4 5 6 7 8 9 \\
        8 & 0 \\
        9 & 0 1 2 \\ \hline
      \end{tabular}
      }} 
  \end{picture}
  \caption{Definition of the multivalued map $G : X \mto X$. The
           left image shows the graph of~$G$. As before, the pentagon
           from Figure~\ref{fig:continuationX} is represented by a line
           segment with identified end points. The table on the
           right lists all function values~$G(x)$.}
  \label{fig:continuationG}
\end{figure}%

On the topological space~$X$, we consider the two multivalued maps
$F : X \mto X$ and $G : X \mto X$ which are defined in the tables in
Figures~\ref{fig:continuationF} and~\ref{fig:continuationG},
respectively. In addition, these two figures show the graphs of these
maps, where we represent the pentagon from Figure~\ref{fig:continuationX}
as a line segment, whose end points correspond to~$0$ and are identified.
Both maps are lower semicontinuous and have closed and acyclic values. In
addition, one can easily see that both maps give rise to a Morse decomposition
with two isolated invariant sets, namely
\begin{displaymath}
  \begin{array}{lcll}
    \DS S_F = \{ 0 \} & \quad\mbox{ and }\quad &
      \DS R_F = \{ 5 \} & \quad\mbox{ for~$F$, and} \\[0.6ex]
    \DS S_G = \{ 0, 1, 2 \} & \quad\mbox{ and }\quad &
      \DS R_G = \{ 5, 6, 7 \} & \quad\mbox{ for~$G$.}
  \end{array}
\end{displaymath}
We claim that the isolated invariant sets~$(S_F,F)$ and~$(S_G,G)$ are related
by an elementary continuation. For this, we use the isolating sets~$N_F
= N_G = \{ 0,1,2 \}$, as well as the topological pair~$P = (P_1,P_2)$ with
$P_1=\{0,1,2\}$ and $P_2=\emptyset$. Then one can easily see that~$P$ is
an index pair for~$S_F$ in~$N_F$ with respect to~$F$, as well as for~$S_G$
in~$N_G$ with respect to~$G$. In addition, the definitions of~$F$ and~$G$
immediately imply $F \le G$, which furnishes our claim. Thus, in view of
Proposition~\ref{prop:continuation} the Conley indices~$C(S_F,F)$
and~$C(S_G,G)$ are isomorphic. We leave it to the reader to verify that
the only nontrivial homology group occurs in dimension zero, that it is
one-dimensional, and that the index map is the identity. In other words,
both isolated invariant sets have the Conley index of an attracting fixed
point. We note that also~$(R_F,F)$ and~$(R_G,G)$ are related by an
elementary continuation, but leave the verification of this and the index
computation to the reader.
\begin{figure} \centering
  \setlength{\unitlength}{1 cm}
  \begin{picture}(14.5,9.5)
    \put(0.0,5.0){
      \includegraphics[height=4.5cm]{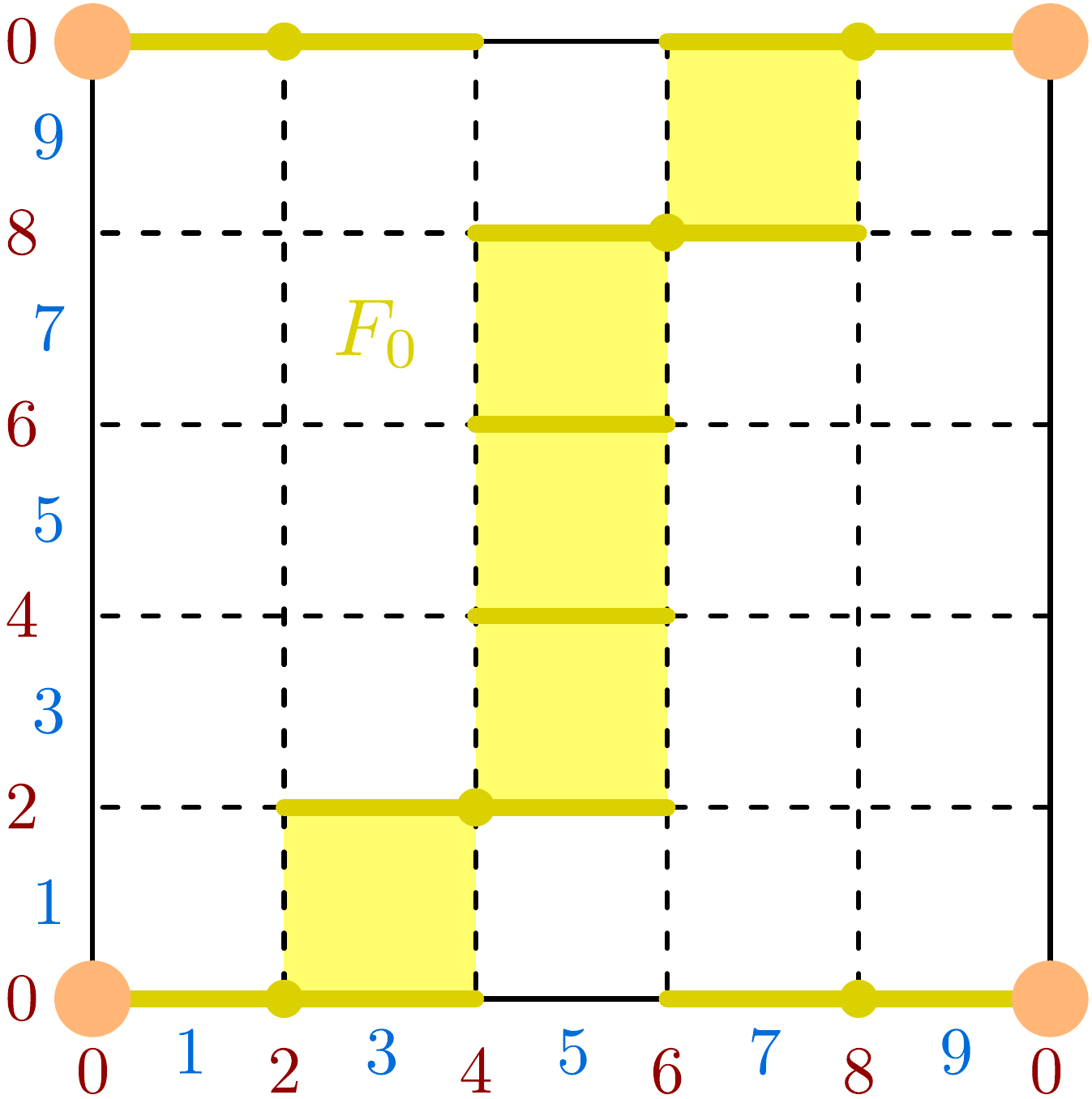}}
    \put(5.0,5.0){
      \includegraphics[height=4.5cm]{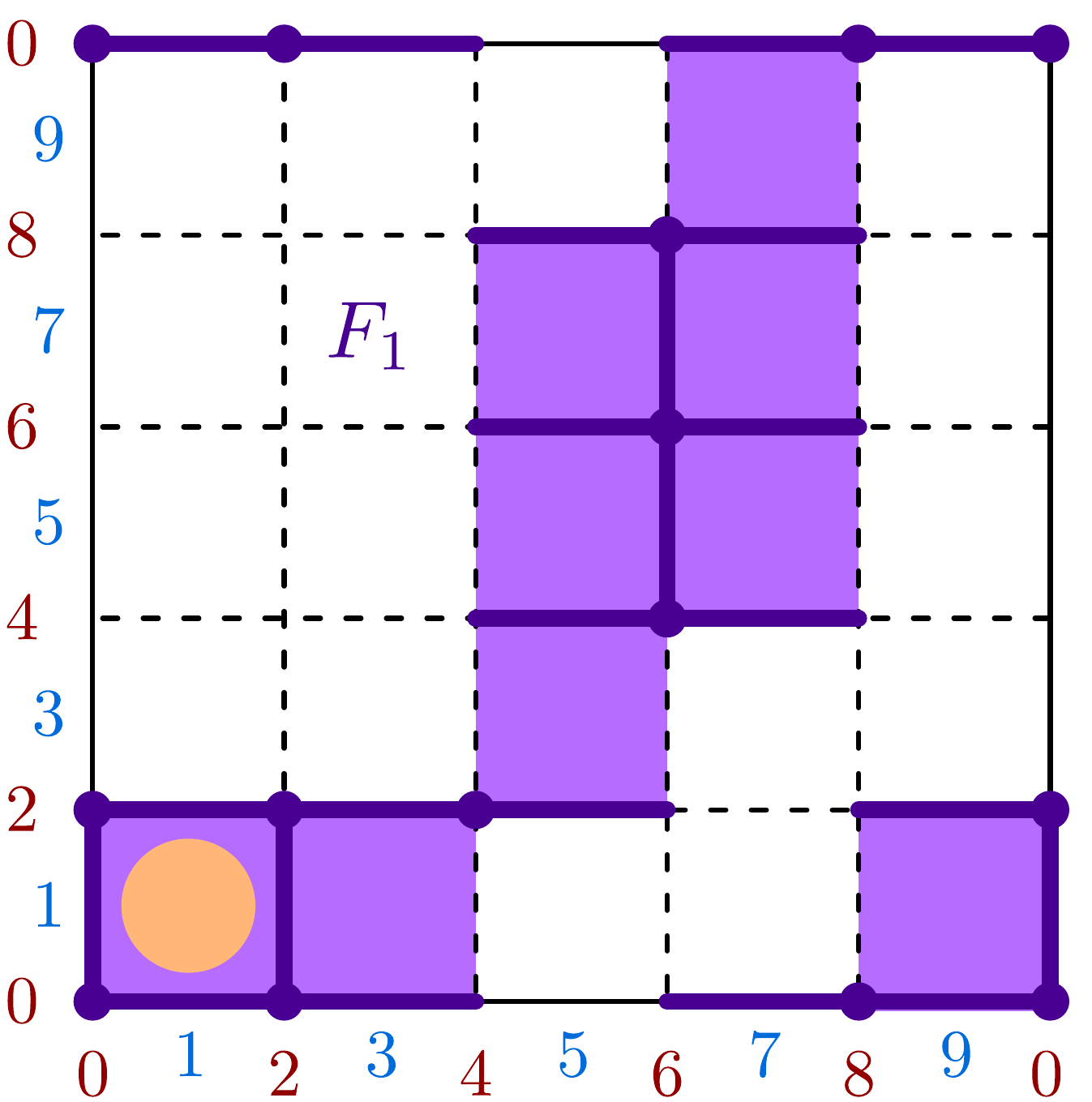}}
    \put(10.0,5.0){
      \includegraphics[height=4.5cm]{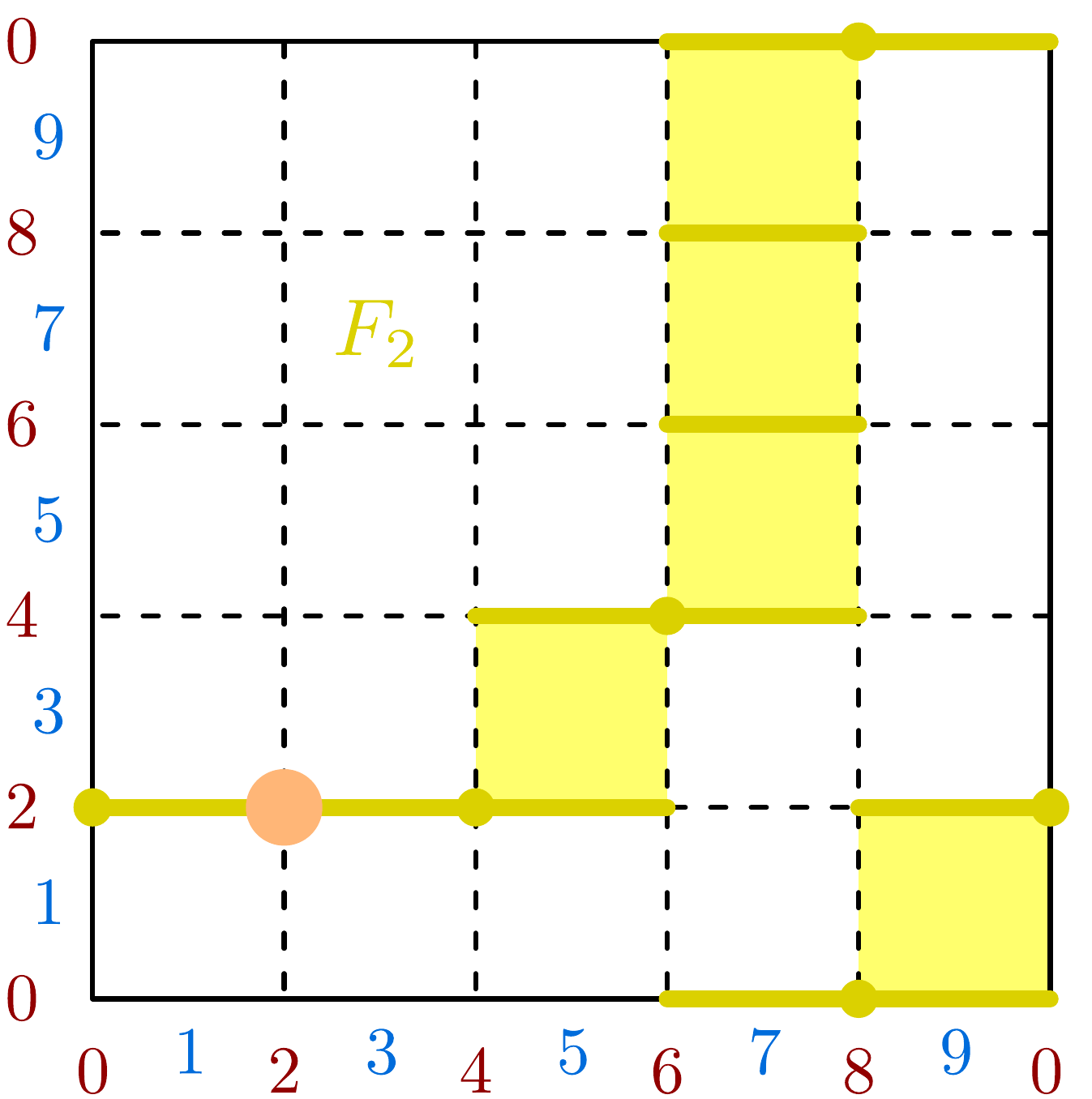}}
    \put(0.0,0.0){
      \includegraphics[height=4.5cm]{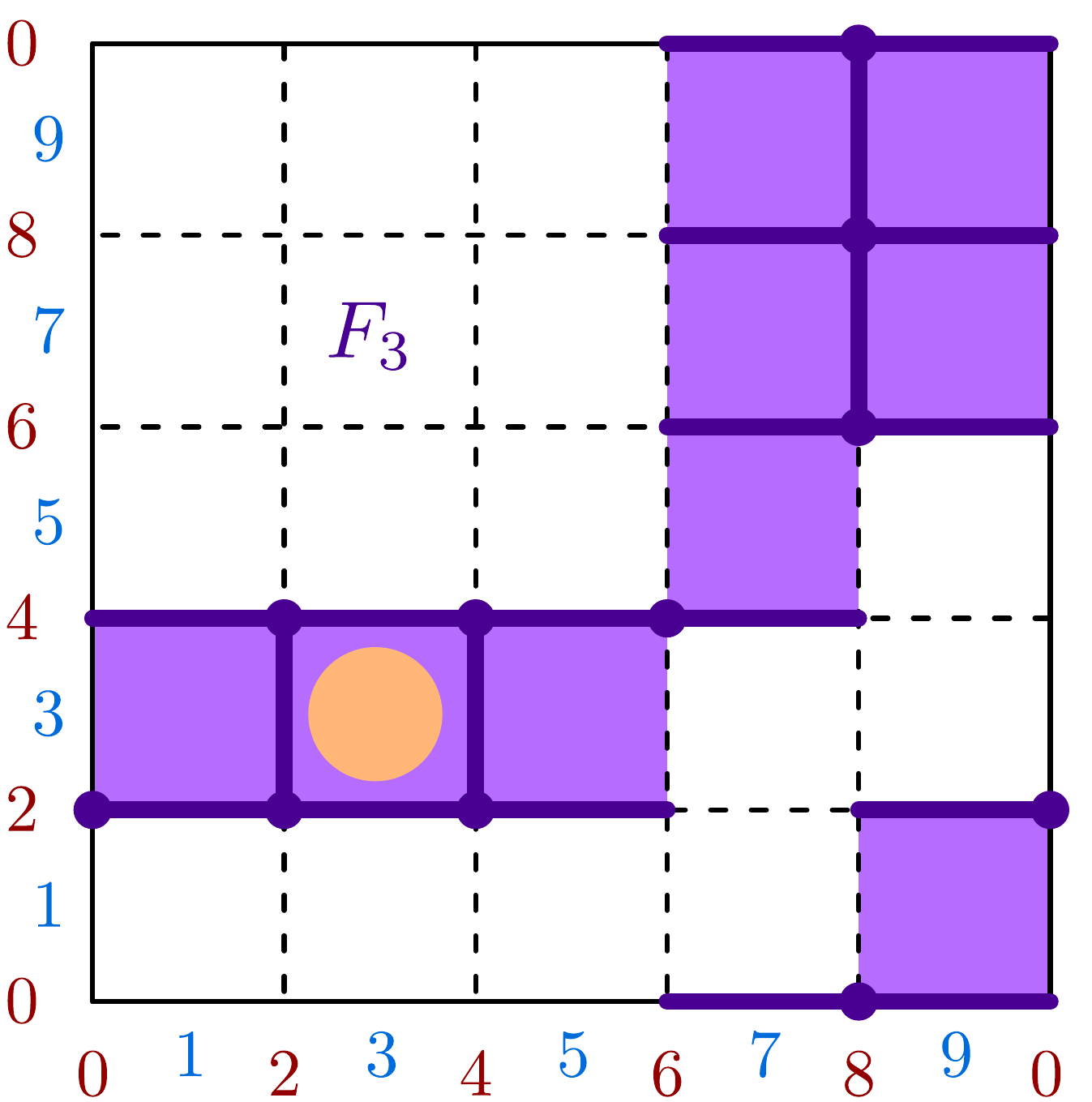}}
    \put(5.0,0.0){
      \includegraphics[height=4.5cm]{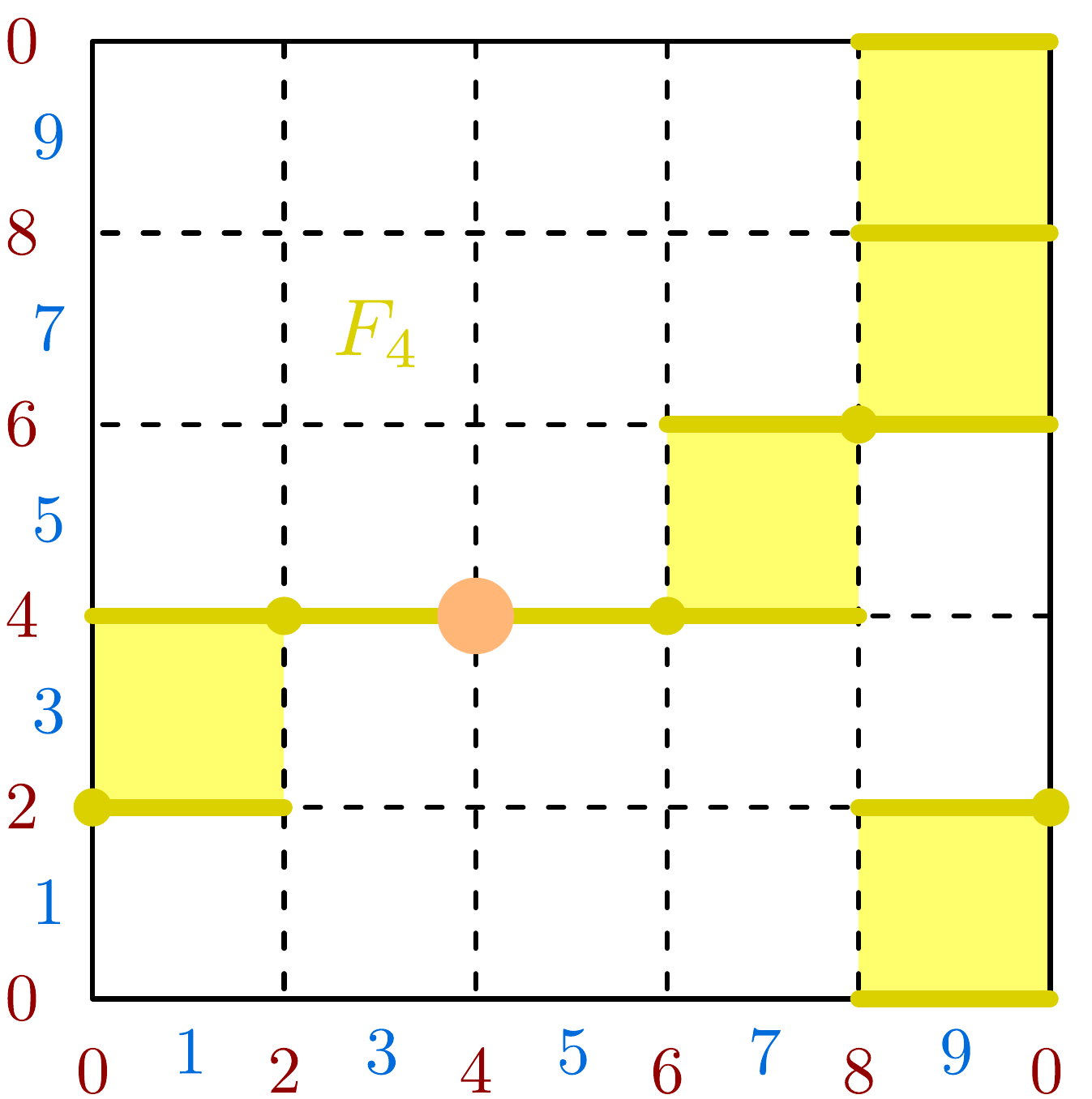}}
    \put(10.0,0.0){
      \includegraphics[height=4.5cm]{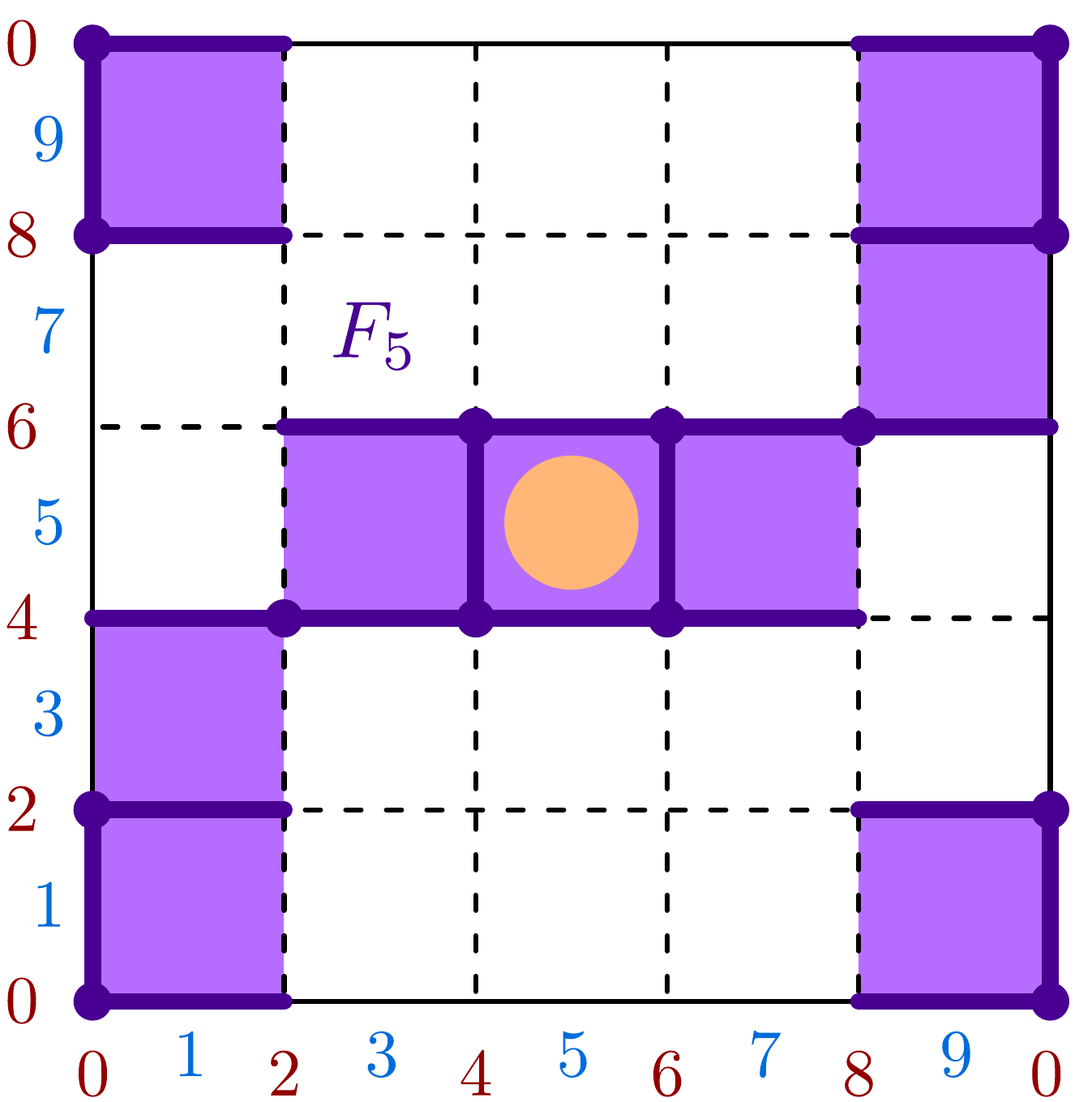}}
  \end{picture}
  \caption{A sample fence $F_0 \le F_1 \ge F_2 \le \ldots$ of
           lower semicontinuous multivalued maps $F_i : X \mto X$
           with closed and acyclic values, as defined
           in~(\ref{eqn:exfence}). The panels depict the first
           six functions of the fence. The associated isolated 
           invariant sets~$S_{F_i}$ are indicated in orange, and 
           they are related by continuation.}
  \label{fig:continuationFk}
\end{figure}%

Yet, even more is true. Recall that we use the representation $X = \ZZ_{10}$
for our underlying topological space~$X$. By using addition and subtraction
modulo~$10$ we can then define the maps $F_i : X \mto X$ via
\begin{equation} \label{eqn:exfence}
  \begin{array}{rclccl}
    \DS F_i(a) & = & \DS F(a-i)+i & \subset & X &
      \quad\mbox{ for even~$i \in \ZZ_{10}$,} \\[0.5ex]
    \DS F_i(a) & = & \DS G(a-i+1)+i-1 & \subset & X &
      \quad\mbox{ for odd~$i \in \ZZ_{10}$,}
  \end{array}
\end{equation}
for every $a\in X$. These definitions give a fence $F_0\le F_1\ge F_2\le
\ldots F_9\ge F_0$ of lower semicontinuous multivalued maps with closed
and acyclic values. By suitably adapting the argument from above, one can
show that for odd~$i$ the map~$F_i$ has the isolated invariant set~$S_{F_i} =
\{ i-1, i, i+1 \}$. Furthermore, this set is related by an elementary
continuation to both the isolated invariant set~$S_{F_{i-1}} = \{i-1\}$
for~$F_{i-1}$, as well as to the isolated invariant set~$S_{F_{i+1}} =
\{i+1\}$ for~$F_{i+1}$. This in turn shows for example that~$S_{F_0} =
\{0\}$ and $S_{F_4} = \{4\}$ are related by continuation. This is
illustrated in Figure~\ref{fig:continuationFk}, where we only depict
the first six functions of the fence, and indicate the isolated invariant
sets in orange.
}
\end{ex}
%
%
%
\section{Future work and open problems}
\label{sec:future}
In this paper, we have developed a notion of isolated invariant sets
and Conley index for multivalued maps on finite topological spaces.
Our theory requires these maps to be lower semicontinuous with closed
and acyclic values. In addition, we have established first properties
of these objects, which mimic the corresponding results in the setting
of classical dynamics. We would like to point out, however, that crucial
assumptions concerning isolation had to be completely changed, due to
poor separation in finite topological spaces. In addition, due
to space constraints, we have omitted a number of properties of the
Conley index, such as for example its additivity, and how it can be
used to detect heteroclinic orbits.

Clearly, the theory presented in this paper is only an abutment of a future
bridge joining combinatorial and classical Conley theory,
aiming at automated algorithmic analysis of dynamical systems. 
The advantage of the theory is its much broader applicability when compared to the 
flow-like combinatorial dynamics of multivector fields~\cite{lipinski:etal:23a}.
However, to get a full bridge, an understanding of formal ties between classical 
and combinatorial theory is needed in the spirit of the results for combinatorial 
vector fields in~\cite{mrozek:wanner:21a,mrozek:etal:22a}.

While the results of this paper are very general and should be useful
in a number of applied situations, we would like to close with a comment
on one unresolved issue. To explain this in more detail, recall that
classical dynamics can be broadly divided into continuous-time and
discrete-time. As we saw earlier in this paper, on finite topological
spaces the continuous-time analogue is trivial. Nevertheless, there
is a dynamical theory which mimics the behavior of flows, and it is
based on the concepts of combinatorial vector and multivector fields,
see~\cite{forman:98a, forman:98b, lipinski:etal:23a, mrozek:17a}. In
these approaches, the flow-like behavior is achieved by requiring
solutions to move between adjacent elements of the space via their
shared boundary. In contrast, the results of the present paper allow
for large jumps in the orbits via iteration of a multivalued map, i.e.,
our results mimic the discrete-time case.
\begin{figure} \centering
  \setlength{\unitlength}{1 cm}
  \begin{picture}(11.0,4.0)
    \put(0.0,0.0){
      \includegraphics[height=4cm]{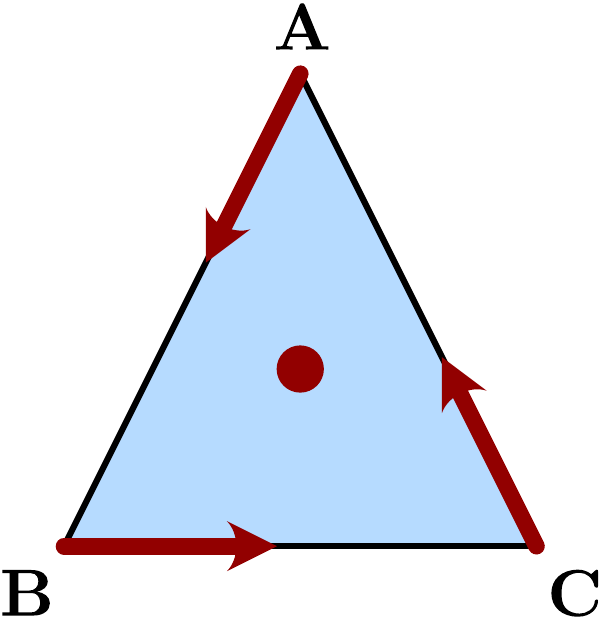}}  
    \put(7.0,0.0){
      \includegraphics[height=4cm]{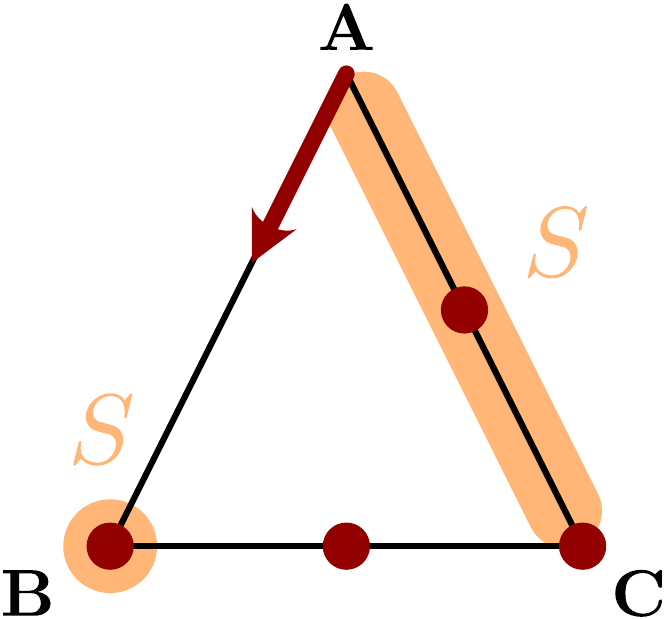}}
  \end{picture}
  \caption{Two sample combinatorial vector fields in the sense of Forman.
           While the one depicted on the left can be represented via an
           admissible multivalued map $F : X \mto X$ on the underlying
           finite topological space with the same overall dynamics, this
           is not possible for the vector field shown on the right.
           There exists no lower semicontinuous $G : Y \mto Y$ with
           closed and acyclic values for which the set~$S = \{ B, AC \}$
           is an isolated invariant set, and such that the map~$G$ has the
           same Morse graph as the indicated combinatorial vector
           field.}
  \label{fig:examplecvf}
\end{figure}%

It is natural to wonder what the relationship is between combinatorial
vector and multivector fields, and the theory of this paper. For classical
dynamics it has been shown in~\cite{mrozek:90b} that every 
isolated invariant set for a continuous-time dynamical system is also
an isolated invariant set for the discrete-time time-one-map. In this sense,
continuous-time dynamical systems can also be studied via discrete-time results.
Is the same true in the case of combinatorial vector fields? To illustrate
this, Figure~\ref{fig:examplecvf} shows two different combinatorial Forman
vector fields. The one on the left is defined on a $2$-simplex, while the
one on the right is defined on a simplicial complex representing the boundary of a triangle.
One can easily see that the dynamics of the left vector field can equivalently
be described by a multivalued map $F : X \mto X$, where~$X$ denotes the 
associated seven-point finite space. One just has to map every vertex to
its opposite edge, every edge to everything along the boundary except itself,
and the triangle to everything --- and the resulting Morse graph induced
by~$F$ is the same as the Morse graph associated with the depicted
combinatorial vector field. However, this is not possible for the example
on the right. If~$Y$ denotes the six-point finite space given by the boundary
of the triangle, then one can show that there exists no lower semicontinuous
multivalued map $G : Y \mto Y$ with closed and acyclic values for which the
set~$S = \{ B, AC \}$ (consisting of a vertex and the opposite edge) is an
isolated invariant set, and such that the Morse graph of~$G$ equals the
Morse graph of the indicated Forman vector field. This failure is due to
our last two requirements on~$G$. It is therefore an interesting open problem
as to whether our theory could be generalized to allow for a larger class of
multivalued maps.

\bibliographystyle{abbrv}
\bibliography{references}

\end{document}